\documentclass[]{imsart}

\RequirePackage{amsthm,amsmath,amsfonts,amssymb}
\RequirePackage[numbers]{natbib}

\startlocaldefs
\numberwithin{equation}{section}
\theoremstyle{plain}

\newtheorem{theorem}{Theorem}[section]
\newtheorem{corollary}{Corollary}[section] 

\theoremstyle{definition}
\newtheorem{definition}[theorem]{Definition}

\def\R{\mathbb{R}}

\endlocaldefs
 
\begin{document}

\begin{frontmatter}
\title{
THERMODYNAMIC FORMALISM \\ 
OF STOCHASTIC \\ 
EQUILIBRIUM ECONOMICS
}
%\title{A sample article title with some additional note\thanksref{T1}}

\runtitle{THERMODYNAMICS OF STOCHASTIC EQUILIBRIUM ECONOMICS}
%\thankstext{T1}{A sample of additional note to the title.}

\begin{aug}
%%%%%%%%%%%%%%%%%%%%%%%%%%%%%%%%%%%%%%%%%%%%%%%
%% Only one address is permitted per author. %%
%% Only division, organization and e-mail is %%
%% included in the address.                  %%
%% Additional information such as            %%
%% identifying the corresponding author must %%
%% be included in in the Acknowledgments     %%
%% section if necessary.                     %%
\author[A]{\fnms{Esa}~\snm{Nummelin}\ead[label=e1]{leena.nummelin@gmail.com}},
\and
\author[B]{\fnms{Elja}~\snm{Arjas}\ead[label=e2]{elja.arjas@helsinki.fi}}

%%%%%%%%%%%%%%%%%%%%%%%%%%%%%%%%%%%%%%%%%%%%%%
%% Addresses                                %%
%%%%%%%%%%%%%%%%%%%%%%%%%%%%%%%%%%%%%%%%%%%%%%
\address[A]{Department of Mathematics and Statistics, University of Helsinki\printead[presep={,\ }]{e1}}

\address[B]{Department of Mathematics and Statistics, University of Helsinki
\printead[presep={,\ }]{e2}}
\end{aug}

\begin{abstract}
In economics, construction of perfect models in a way that would be comparable to the standards customary in physical sciences is generally not feasible. In particular, the observed value for an economic equilibrium may deviate
significantly from its model-based a priori expected value. Mathematically, the a posteriori   observed equilibrium
may then represent a large deviation in the sense that it falls outside the region
of validity of the Central Limit Theorem.
With this as the motivating starting point, we propose a  new approach to the theory of stochastic
economic equilibrium. Drawing on recent developments in probability
theory, we argue for the relevance of the  theory of  large
deviations in stochastic equilibrium economics.
Thereby the formalism of stochastic
equilibrium economics becomes analogous to that of classical
statistical mechanics, as the 
theory of large deviations forms also the mathematical basis of
statistical mechanics. In consequence, thermodynamic concepts such as entropy, partition function and canonical probability can be introduced in a natural way to stochastic equilibrium economics. 
We focus here on the economic analogs of two fundamental principles, the  Second Law of Thermodynamics and the Gibbs Conditioning Principle.
\end{abstract}



\end{frontmatter}
%%%%%%%%%%%%%%%%%%%%%%%%%%%%%%%%%%%%%%%%%%%%%%
%% Please use \tableofcontents for articles %%
%% with 50 pages and more                   %%
%%%%%%%%%%%%%%%%%%%%%%%%%%%%%%%%%%%%%%%%%%%%%%
%\tableofcontents

%%%%%%%%%%%%%%%%%%%%%%%%%%%%%%%%%%%%%%%%%%%%%%
%%%% Main text entry area:

\section{Introduction}

\label{sec1}

\subsection{The set-up}

\label{sub11}

The ideal in modern theoretical economics  is to arrive at a correspondence between the considered models and reality that is present, and  
paradigmatically characteristic, to   physical sciences.
 
Classical  equilibrium theory is concerned with the following
fundamental problem:

\medskip

\textsc{Problem 0}. Does there exist an equilibrium, i.e., a price vector at
which the total excess demand in the economy vanishes?

\medskip

In the classical  theory the total excess demand is  a known deterministic function (e.g., Debreu \cite{deb}), and 
equilibrium prices are defined as zeros of this function. In reality, excess demand is not known, but it is still implicitly thought that the  equilibrium
prices are zeros of an hypothetical total excess demand function.

Following the pioneering works by Arrow, Debreu and McKenzie in the 1950s
(\cite{arr}, \cite{arrdeb}, \cite{deb}, \cite{mck}), the theory of
general (deterministic) equilibrium became a major research topic, and was then
developed further during the 1960s and 1970s, to reach its present mature state.

As all human behaviour, economic behaviour is to some extent unpredictable, and often irrational. In contrast to physical sciences, in economics it is never possible
to achieve perfect correspondence between model and reality. An economic model, however sophistically built, is
only an approximate representation of the real world, with the discrepancies between model and reality corresponding to a priori uncertainty.

% The model can be tested by making
%"objective"  a posteriori  observations on t%the real economic system that the
% model aims to describe. 

\if 0 Due to the imperfectness,
the  observations on the real  economy may well contradict
with the predictions by the apriori model, with these   contradictions  indicate  the  presence of   a priori
uncertainty. \fi

Accounting  for uncertainty leads to consideration of an entire ensemble of
configurations, where each configuration is combined with a weight (a priori
probability) indicating the  a priori degree of belief in its truth.
These  beliefs can be   "rational guesses"  or be based, e.g., on an earlier systematic statistical study of the microeconomic behaviour of a suitably chosen representative
sample drawn from the set of economic agents. These beliefs
can be described  mathematically by specifying an a priori probability law.

In view of this, it is natural to consider equilibrium
models, which are  stochastic  rather than deterministic. For this, we assume that the  parameters in the equilibrium
model are  random variables, and thereby also 
the equilibria of the model are  random. 

We also assume that the economy is {\it large}, i.e., the number of participating agents is large. The question of the  existence of an equilibrium, with probability $1$, in random economies was famously studied in Hildenbrand \cite{hil}.

The first special problem one can pose
for a large random economy is the following:
\medskip

\textsc{Problem 1.} What is the a priori relationship of  a random equilibrium  with
its model-based a priori expected value? More precisely,  does there exist a Law of Large Numbers (LLN) and a Central Limit Theorem (CLT)?
\medskip

The seminal work dealing with the solution to Problem 1 is due to Bhattacharya and Majumdar 
\cite{bhamaj1}. 
In contrast to deterministic equilibria, after
the pioneering  papers by Hildenbrand, Bhattacharya and Majumdar in the 
early 1970s,  interest in the  theory of random economic equilibria has remained  marginal. Drawing on recent developments in general probability theory, we suggest here a new  approach to this
theory, by arguing for the important role of the
theory of large deviations in this context.

We formulate two basic problems concerning the Bayesian type of
interplay between an  a priori random equilibrium model with the
a posteriori observations on the realized equilibrium.

By an {\it a posteriori (macroeconomic)   observation} we mean  the
observation of some   realized macroeconomic variable or quantity such as,
e.g, the equilibrium prices, the consumption or production of
some commodities in the  economy or in  some economic sector, or the
number of economic agents being of a certain "type". We focus on the case where  the observations  concern the  
realized   equilibrium prices in the system.

The first new problem to be addressed  is as follows:

\medskip
\textsc{Problem 2.} Suppose that the realized equilibrium
is a posteriori observed.
What is the  a priori probability of this observation?

\medskip
Note that, if the a priori model is "good", then, due to the LLN and
CLT, this probability ought to be "high". Moreover, in this case its
value can be approximated, due to CLT, see  Section \ref{sub25}.

However, since economic models are not perfect, the a posteriori observed equilibrium may deviate
significantly from its model-based a priori expected value, in which case the a priori probability of the observed value is  small.  In fact, it can be argued that the observed equilibrium then represents a  Large Deviation (LD) in the sense that it falls
outside the region of the validity of the CLT. This is so because
the  CLT  is valid only in an asymptotically small region, with standard deviation of
order ${1 \over \sqrt{n}}$, cf.
Section \ref{sub25}.

We then address the  inference problem arising from the a posteriori
observation of the equilibrium (possibly representing a large
deviation):

\medskip
\textsc{Problem 3.} What is the  a posteriori probability law governing a
random economy, conditionally on an 
observation of the realized  equilibrium?

\medskip
This inference problem arises also in the classical non-stochastic
equilibrium theory. Namely, due to the (unavoidable)
contradictions of   the a posteriori observations on the real system
 with their model-based predictions, some degree of re-modelling becomes necessary.
 In the deterministic case, particular versions of such models remain necessarily ad hoc.

 \medskip

 \textsc{Remark.} These ideas have certain similarity, but also obvious differences, to those familiar from the Bayesian statistical inference and decision theory. There, a postulated prior distribution of the model parameters is updated, based on the empirical observations (data) and by applying Bayes' formula, into the corresponding  posterior distribution. In each case,   predictive probabilities of future observables can be computed by an integration of the corresponding likelihood with respect to  the distribution of the model parameters. 

\subsection{Short description of the results}

\label{sub12}
Problems 2 and 3 have close  analogs in the
formalism of  classical statistical mechanics. This is because the theory of large deviations can be seen to have a similar role in statistical mechanics as in stochastic equilibrium economics.

\medskip
The set-up of Problem 2 turns out to be an analog of the Second Law
of thermodynamics. To this end, recall that the (integral form of the)
classical Second Law relates thermodynamic entropy with temperature,
internal energy, and thermodynamic partition function, see
(\ref{equ216}). 

We argue that in stochastic equilibrium economics   {\it information content}, defined as the logarithm of the inverse probability of the observation
of an equilibrium, should be viewed as the   analog of thermodynamic entropy in statistical mechanics. The Second Law of
stochastic equilibrium economics then relates economic entropy to the corresponding
economic partition function.

The economic Second Law is obtained at the ideal limit of an infinitely
large economy, the analog of the thermodynamic limit, from a Theorem of Large Deviations (TLD) concerning
the random equilibrium prices. The TLD gives a formula for the a priori probability of the
observation of an a priori unexpected equilibrium. 
The economic Second Law can then be interpreted as providing an information theoretic  measure of goodness for the a priori equilibrium
model.

\medskip
The set-up of Problem 3 is an analog of the thermodynamic Gibbs
Conditioning Principle (GCP) characterizing  the  thermodynamic
canonical probability law as the governing probability law of a
thermodynamic system at a measured temperature. The canonical law is a
member of the exponential probability distribution family
generated by the total energy, see Section \ref{sub31}.

According to the  economic analog of the Gibbs Conditioning Principle,
conditionally on the observation of an equilibrium, the a posteriori
probability law governing a random economy is a canonical member of
the exponential family  generated by the random total excess demand
(Section \ref{sub32}). GCP is a version of a  conditional Law of
Large Numbers concerning  macroeconomic random variables (Theorem \ref{Theorem 3.}
in Section \ref{sub52}). \medskip

We organize the paper into two parts. In the first part we develop the
thermodynamic formalism of large random economies, focusing on the
economic analogs of the Second Law of thermodynamics and the Gibbs
Conditioning Principle. In the second part  we formulate the economic Second Law and the Gibbs Conditioning
Principle as mathematical theorems within the framework of the
theory of LD.

As illustrations, we investigate  in  detail  the special case of an {\it  ideal
random economy}, where  the economic agents are supposed to be
statistically identical and independent, being an analog of the
classical thermodynamic ideal gas. As a special illustrating
example we study   "ideal" random Cobb-Douglas economies.

There do not seem to be no easily accessible treatments of the LD theoretic
foundations of statistical mechanics in the literature. Therefore, for the convenience of
a possibly unacquainted reader and in order to point out the proposed
analogy of stochastic equilibrium economics with statistical mechanics,
we 
begin each chapter with a short review on the relevant basics of
statistical mechanics and its LD theoretic foundations.

We have made an attempt to keep the presentation self-contained, thereby
avoiding references to the general LD theory.  For a reader
interested in this theory we recommend the monographs by
Bucklew \cite{buc} and by Dembo and Zeitouni \cite{demzei}. 

\if 0
We plan to study  {\it stochastic finance markets} and the so-called
{\it survival model} as applications of the proposed formalism in the
forthcoming papers \cite{num5}, \cite{num6}.\fi

\bigskip

%\bigskip\bigskip\bigskip

\begin{Large}

\noindent PART I: THERMODYNAMIC FORMALISM

\end{Large}

\section{The Second Law}
\label{sec2}

\subsection{The Second Law of thermodynamics}
\label{sub21}

The proposed formalism for stochastic equilibrium theory  is analogous
to that of statistical mechanics.
Therefore, in order to point out 
this analogy,
we review in Sections 2.1 and 3.1 the classical Second Law of thermodynamics and the Gibbs Conditioning Principle.
 
Consider a physical system which comprises  $n$ {\it particles} $i=1,2,...,n$.
Their  positions $\rho_i \in \R^3$ and momenta $\theta_i \in \R^3$  
form a {\it particle configuration}
$\omega$ in the {\it thermodynamic
ensemble} $\Omega \ \dot=\ \R^{6n}$. Associated with each
particle configuration $\omega$ there is the {\it energy}
$U(\omega) \ (\dot= $ the sum of the kinetic energies of the particles and of
the potential energy associated with the particle configuration
$\omega$).

According to Liouville's theorem,  Lebesgue measure $d\omega$ (= the Euclidean volume)
in $\R^{6n}$ is invariant under the Hamiltonian dynamics (see
\cite{mar1}: Chapter 1). This
means that  Lebesgue measure can be regarded as the "a priori
probability law".
In statistical mechanics the a priori model is "completely
imperfect" in that,
a priori, all configurations $\omega$ are equiprobable.
 
%KAAVAT VOIDAAN NUMEROIDA EHKA PAPREMMIN KAHDELLA NUMEROLLA, JOISTA ENSIMMAINEN ON 'SECTION NUMBER' JA TOINEN JUOKSEVA NUMERO. MITEN SEN SAA AIKAAN?
\medskip

Observation (measurement) of the {\it temperature}  $T$  restricts the
thermodynamic system to a compact 
{\it  energy shell}:
\begin{equation}
\label{equ210}
  \{  |U-E| < \Delta\}\ \dot=\ \{\omega \in \Omega:\  |U(\omega)-E| < \Delta\}.
\end{equation}
Here $E $ denotes the internal energy at temperature $T$, and
$\Delta$ denotes the thickness of the
"infinitesimally thin" energy shell. (The symbol "dot" indicates equality by definition.)

The {\it thermodynamic entropy} $S$  is defined
as the logarithm of the volume of this  energy shell:
\begin{equation}
\label{equ211}
S\ \dot=\ \log \hbox{Vol} \{ |U-E| < \Delta\},
\end{equation} 
(\cite{pliber}: p.32).

The {\it thermodynamic partition function} $\Lambda(\beta)$ is defined as the
Laplace transform of the  energy:
\begin{equation}
\label{equ212}
\Lambda(\beta)\ \dot=\ \int e^{-\beta U(\omega)}d\omega,\ \beta >
0,\end{equation}
where the  conjugate variable
$\beta$ has the meaning of {\it inverse 
temperature}: $T \ \dot=\ {1 \over
\beta}$. 
The  {\it internal energy} $E(\beta)$ associated with the inverse
temperature $\beta$ is defined as the derivative
\begin{equation}
\label{equ213}
E(\beta)\ \dot=\ -{d \over
d\beta} \log \Lambda(\beta).
\end{equation}

%**** Darion kommentti: "Shall we say that this equals $\langle U(\beta) \rangle$ where accoding physics notations we use angle bracket for expectation?"

In physics, this quantity is traditionally denoted by $\langle U(\beta) \rangle$. 
\medskip

The Second Law of thermodynamics introduces  entropy as an extensive thermodynamic variable. According to this law, an infinitesimal reversible addition $dQ$ of heat leads to a proportional increase in entropy, with the inverse temperature as the coefficient of proportionality: 
\begin{equation}
\label{equ214}
dS=\beta dQ.
\end{equation}
%see  \cite{pliber}: Section 1.2.2. 
For a system in constant volume (doing no work), added heat adds influences solely the internal energy of the system. Therefore, for a system of constant volume we have
\begin{equation}
\label{equ215}
dS=\beta dE.
\end{equation}
              
Integrating by parts in (\ref{equ215})
 and taking into account  
(\ref{equ213}) leads to the equivalent 
integral form of the second law:
\begin{eqnarray}
\label{equ216}
S(\beta)& = & \beta E(\beta)-\int E(\beta) d\beta \nonumber \\ 
& = & \beta E(\beta)+ \log \Lambda(\beta). 
\end{eqnarray}
The Second Law is obtained at the ideal limit
$n = \infty$ (the so-called {\it thermodynamic limit})
from a theorem of large deviations concerning
the total energy $U$, see Section \ref{sub41}.
\bigskip

\textsc{ Example: The classical ideal gas.}
In the classical ideal gas there is no interaction between the
particles so that the energy comprises  solely the kinetic energies of the
individual particles, see \cite{mar1}: Section 2.1. Therefore, it is sufficient to include in the ensemble only  the momenta of the particles: $$\Omega \ \dot=\ \{\omega  =
(\theta_1,...,\theta_n):\ \theta_i \in \R^3\}=R^{3n}.$$
The kinetic energy $u_i$  of particle $i$  is
given by the "structure function"
$$u_i = u (\theta_i) ={|\theta_i|^2 \over 2m} $$ of its momentum $\theta_i$
("thermodynamic characteristic") and  
mass $m$. (The non-standard 
terms "structure function" and "thermodynamic characteristics"
refer to their economic analogs, cf. Section \ref{sub24}.)
Thus the total energy
becomes simply the sum
\begin{equation}
\label{equ217}
U(\omega) =  \sum\limits^n_{i=1}u(\theta_i) = \sum\limits^n_{i=1}
{|\theta_i|^2 \over 2m} .\end{equation}

%\medskip
%**** Darion kommentti: "Would it be better to have an explicit subindex $n$ for all quantities which depend on $n$ ? Or at least add a comment like "we omit the subindex $n$ for all quantities which are dependent on the system size"."
Here, and in numerous places elsewhere, we omit   the dependence   on the system size $n$ in the notation.

Consequently, the partition function of the
ideal gas is the $n$th power
of the partition function associated with a single particle:
\begin{equation}\label{equ217a}
\Lambda(\beta) \dot= \int \cdots \int e^{- 
\beta\sum\limits^n_{i=1}
{|\theta_i|^2 \over 2m}} d\theta_1 \cdots d\theta_n = \lambda(\beta)^n,\end{equation}
where
\begin{equation}
\label{equ218}
\lambda(\beta)\ \dot=\ \int\limits_{R^3} e^{-\beta u(\theta)}d\theta =
\int\limits_{R^3} e^{-{\beta |\theta|^2 \over 2m}}d\theta  =
(\int\limits_R e^{-{\beta x^2 \over 2m}}  dx)^3 = (2 \pi m)^{{3 \over
2}}\beta^{-{3 \over 2}}.\end{equation}

The internal energy of the ideal gas is given by
\begin{equation}
\label{equ218a}
E(\beta)= n e(\beta),\end{equation}
 where
\begin{equation}
\label{equ219}
e(\beta)\ \dot=\ -{d \over d\beta} \log \lambda(\beta)  = {3 \over 2\beta}
\end{equation}
denotes the internal energy of a single particle. Therefore, for the ideal gas,
 the integral form (\ref{equ216}) of the Second Law   becomes
\begin{equation}
\label{equ2110}
S(\beta) = n( \log \lambda(\beta)+\beta e(\beta))  =
 {3n \over 2}(- \log \beta + \log 2 \pi em).
\end{equation}

\subsection{Random economies and their equilibria}
\label{sub22}
\medskip

We consider 
an {\it   economic system} (shortly, {\it economy}), where
certain  {\it  commodities} $j=1,...,l+1$
are traded  by a set of {\it economic agents.}

We assume that there is  a parameter $n$, to be  called  the
{\it size parameter}. (Typically,  $n$ is simply  the  number of
economic agents.) We assume that we are dealing with a "large
economy"; namely, in the exact 
theorems we let $n\to \infty$.

Let $Z^j(p)$ denote the {\it total excess demand} on  commodity $j \in \{1,...,l+1\}$ at  {\it price} $p \in \R^{l+1}_+$. (Superscripts refer to  commodities.) We assume that, for each $j$ and $p$,  the total excess demand is a  random variable defined on an underlying  probability space $(\Omega,{\cal F},P)$, i.e.,  $$Z^j(p)=\{Z^j(\omega;p);\ \omega \in \Omega \},$$ where $Z^j(\omega;p)$ is a measurable map of the variable  $\omega$, see e.g. \cite{bil}: p. 182.

We  refer to $\Omega$ as the {\it macroeconomic ensemble} and to its
elements $\omega$  as the {\it macroeconomic configurations}. (This somewhat peculiar terminology is
due to  the analogy with statistical mechanics, cf. Section \ref{sub21}.)
The underlying probability measure $P$ is called the {\it
a priori
macroeconomic probability law}. 

We  make the following two standard assumptions:

\medskip
\noindent (i) $ Z^j(ap) \equiv Z^j(p)$ for every  constant
$a>0\ $ (homogeneity of degree 0); and

\medskip
\noindent (ii) $\sum\limits^{l+1}_{j=1}p^j Z^j(p) \equiv 0\ $  (Walras' law).

%\medskip
\noindent  As is common,  the symbol $\omega$ is omitted.
Due to the homogeneity of degree 0,  the prices can without loss of
generality be normed to belong to  the  {\it price simplex}
$$S^l \ \dot=\ \{p \in \R^{l+1}_+:\ \sum\limits^{l+1}_{j=1}p^j = 1\}.$$
%(The symbol "dot"  indicates equality by definition.)

Walras' law implies  that, for any price $p \in S^{l}$ such that $p^{l+1} >
0$, the total excess demand on the $(l+1)$st commodity is determined by
the total excess demands on the other commodities:
$$Z^{l+1}(p) = -(p^{l+1})^{-1} \sum\limits^l_{j=1}p^jZ^j(p).$$
Thus we  omit the  $(l+1)$st component and call the vector
 $$Z(p)\ \dot=\ (Z^1(p),...,Z^l(p)) \in \R^l$$  comprising
the total excess demands on the  commodities $j= 1,...,l$
simply the {\it (random) total excess demand}.

\medskip
The {\it random equilibrium  prices (r.e.p.'s)}  are defined as those price vectors $p^*=\{p^*(\omega);\ \omega \in \Omega \}$ at which the random total excess demand vanishes: \begin{equation} \label{equ220} Z(\omega;p^*(\omega)) = 0,\end{equation} or, shortly: $$Z(p^*) = 0.$$

Let $$EZ(p)\ \dot=\ \int Z(\omega;p)P(d\omega)$$
denote the (deterministic) {\it expected total excess demand function.}
Its zeros $p^*_e$ are called  the {\it (a priori) expected
equilibrium prices}:
$$EZ(p^*_e)= 0.$$
\medskip
\textsc{Remarks:}
%\medskip
(i) The random equilibrium prices form a  random set:
$$\pi^*= \{\pi^*(\omega); \omega \in \Omega\},$$
where 
$$\pi^*(\omega)= \{p \in S^l:\ Z(\omega;p)=0\}$$
denotes the set of equilibrium prices at the realized macroeconomic configuration $\omega$.
Thus
$p^*(\omega)$
denotes an arbitrary element of 
$\pi^*(\omega)$.

\medskip
(ii)
There is  a more general concept of random equilibrium: 
A  random variable
$X(p)= \{X(\omega;p);\ \omega \in \Omega\} \in \R^d$ (for some $d \geq 1$),
which depends on the macroeconomic configuration $\omega$ and on the price $p$,  is called a {\it macroeconomic random variable}. Examples of such variables are total demand, supply, production or share of these by the whole economy or by some macroeconomic sector. Also, if the economic agents can be classified into a finite set of different {\it types,} then the  numbers of agents belonging to these classes can be regarded as macroeconomic random variables, cf. Example (iii) below.
\medskip

For any r.e.p. $p^*$, let $$x^* \ \dot=\ n^{-1}X(p^*)$$ 

denote the  mean  of the macroeconomic random variable at this equilibrium price. 
\medskip

%**** Darion kommentti: "At this stage it is not completely clear what is $n$,  (it would be better to have an $n$ index for the other variables in this and other expression). Also is the expectation missing?" 

%**** Tähän asiaan tarvittaisiin Esan ehdotus! 
%\medskip

The  pair $(p^*,x^*)$  is called a {\it random composite equilibrium (r.c.e.)}, see \cite{num4}: Section 4.3.

\medskip
\textsc{ Examples: }
%\medskip
(i) In the standard Cobb-Douglas exchange economy (shortly, 
CD economy) comprising $n$
economic agents $i=1,...,n,$ the {\it individual excess demand} by agent $i$ is
given by the vector
\begin{equation}
\label{equ221}
 \zeta_i(p)\ \dot=\ \bigl( \ {a_i^j \over p^j} p \cdot e_i
- e_i^j;\ j=1,...,l \bigr),\end{equation} where
$a_i\ \dot=\ (a^1_i,...,a^{l+1}_i) \in S^l$ is the
{\it share parameter},
$e_i\ \dot=\ (e^1_i,...,e^{l+1}_i) \in \R^{l+1}$ is the  {\it initial
endowment}, and $ p \cdot e_i \ \dot=\  \sum\limits_{k=1}^{l+1} p^k
e_i^k$ is the  {\it initial wealth} of the agent $i$.

In a  random CD economy  the share parameters $a_i$ and the
initial endowments $e_i$ are 
random variables.

The  random total excess demand equals the sum of the random individual
excess demands $\zeta_i(p)$:
$$Z(p)\ \dot=\  \sum\limits_{i=1}^n \zeta_i(p) = \biggl(
\frac{1}
{p^j}
\sum\limits_{i=1}^n\sum\limits_{k=1}^{l+1} a^j_ie^k_i p^k -
\sum\limits_{i=1}^ne^j_i;\ j=1,...,l\biggr).$$
The random equilibrium prices are then obtained from  
\begin{equation}
\label{equ221a}
p^* = \biggl
(\biggl(\sum\limits_{i=1}^ne^j_i\biggr)^{-1}(W^*)^j;\ j=1,...,l+1\biggr),\end{equation}
where $W^* = ((W^*)^1,...,(W^*)^{l+1})$ is a left eigenvector
(associated with the eigenvalue $1$) of the (random) stochastic
matrix
 \begin{equation}
\label{equ221b}
A\ 
= (a^{jk};\ j,k=1,...,l+1)\ 
\dot = \,\biggl(\biggl(\sum\limits_{i=1}^ne^j_i\biggr)^{-1} \sum\limits_{i=1}^n
 a^k_ie^j_i;\ j,k=1,...,l+1\biggr),\end{equation}
 normalized so that $p^* \in S^l$. (In particular, it follows that, if
 the  matrix $A$ is {\it
irreducible} (e.g., \cite{sen}), then
there  is only one (unique) random equilibrium price $p^*$.)

The  expected total  excess demand on
commodity $j$ is given by
$$EZ(p) =  \sum\limits_{i=1}^n E\zeta_i(p)
 = \,  \biggl( \frac{1}{p^j}\sum\limits_{k=1}^{l+1} M_{a;e}^{jk} p^k -
M_e^j;\ j=1,...,l \biggr),$$ where
\begin{eqnarray}
M_{a;e}^{jk}  & \dot= &
 \sum\limits_{i=1}^nE(a^j_ie^k_i), \nonumber \\
%= \int \int a^je^k f(a,e)dade, 
M_e^k & \dot = & \sum\limits_{i=1}^nEe^k_i = \sum\limits_{j=1}^{l+1} \mu_{a;e}^{jk}. \nonumber
\end{eqnarray}
The  expected equilibrium prices are  given by
\begin{equation}
\label{equ222}
p_e^*= ((M_e^j)^{-1}(w_e^*)^j;\ j=1,...,l),\end{equation}
 where
$w_e^*$ is a left eigenvector with eigenvalue $1$  of the stochastic matrix $$A_e 
= (a_e^{jk};\ j,k=1,...,l+1)\ \dot=\ 
((M_e^j)^{-1}M_{a;e}^{kj};\ j,k=1,...,l+1),$$  subject to the
normalization  $p^*_e \in S^l$.
Again, if  $A_e$ is irreducible, then
there  is only one (unique) expected equilibrium price $p^*_e$.

\medskip

(ii) In a {\it random (one-period)  financial market} (\cite{num4}: Example 2) the parameters  $a_i \ (\dot=\ $ the risk aversion of agent $i$), $\mu_i \ (\dot=\ $ the vector of subjective expectations by $i$ of the values of the assets at the end of the period), $\Sigma_i \
(\dot=\ $the matrix  of subjective expectations by $i$ of the
correlations of the values of
the assets at the end of the period) and $e_i \
(\dot=\ $ the initial endowment of $i$) are  random variables.
(There is "double stochasticity" in that the agents' subjective
expectations $\mu_i$ and $\Sigma_i$ are also regarded as
random variables.)

\medskip

(iii) The concept of a  random composite equilibrium
can be illustrated with the following {\it survival model} (see e.g. \cite{bhamaj2}). 
Consider a  random Cobb-Douglas economy as described above. Suppose  that at each price $p$ there is a {\it survival level} $\underline w(p)$
such that an agent $i$ having initial endowment $e_i \in \R^{l+1}$ can survive only if his initial wealth exceeds this
level, i.e.,
$p \cdot e_i \geq \underline w(p)$. Let $\chi_{\{p \cdot e_i < \underline w(p)\}} $  denote the indicator of
non-survival, i.e., it is $1$ if the inequality  $p \cdot e_i < \underline w (p)$ is satisfied and $0$ otherwise.
Thus we can express
the total number $N(p)$ of non-surviving agents as the sum
$$N(p) = \sum\limits_{i=1}^n  
\chi_{\{p \cdot e_i < \underline w(p)\}}
.$$

\medskip
%**** Darion kommentti: "at this stage this is not clear how the survival affects the model,  perhaps it needs a more detailed explanation."

%**** Tähän tarvittaisiin Esan ehdotus!

%\medskip

The pair $(p^*,n^*)$
comprising the equilibrium price $p^*$
and the mean number $n^* \dot= n^{-1}N(p^*)$
of non-surviving agents (at this equilibrium) is now an example of 
a random composite equilibrium.
%\medskip

\subsection{The Second Law of stochastic equilibrium economics}

\label{sub23}
 
Let $p$ be any price belonging to the price simplex $S^l.$

By the  {\it observation} of an  equilibrium price {\it  in the
$\delta-$neighborhood}  of  price  $p$ we mean  the observation of  a
random equilibrium price, which is at a distance less than
$\delta$ from  $p$, i.e., the occurrence of the
event
\begin{equation}
\label{equ231}
\{\omega \in \Omega:\ \hbox{there exists a r.e.p.}\ p^*(\omega) \ \hbox{such
that}\ |p^*(\omega)-p| < \delta\}.
\end{equation}
We think of (\ref{equ231}) as the observation
of the realized ("true") equilibrium price in the
$\delta-$neighborhood of price $p$.
The "tolerance" $\delta $ is supposed to be negligible,    and
therefore we may speak about the   observation of the
equilibrium price  {\it at}   price $p$. In the sequel we 
write the observation
(\ref{equ231}) shortly as 
\begin{equation}
\label{equ231a}
\exists p^*:|p^*-p|< \delta.\end{equation}
In the ensuing exact
theorems we let  $\delta \to  0$.

\medskip

Let $A$ be any event having probability $P(A)$. The {\it 
information content} ${\cal I}(A)$ in the observation of the event $A$
is defined as the logarithm of the inverse of its
probability:
$${\cal I}(A)\ \dot=\ \log {1 \over P(A)} = -\log P(A)$$  (e.g., \cite{covtho}). Thus the observation of a "common" event having high probability has low information content  whereas the observation of a "rare" event has high information content.

\medskip
%\textsc{ Definition 1.}
\begin{definition}\label{def1}
The {\it economic entropy}  $I(p)$ is defined
as the {\it information content} in  the observation of a random
equilibrium price at  price $p$:
\begin{equation}
\label{equ232} 
I(p)\ \dot=\  {\cal I}(\exists
p^*: |p^*-p| < \delta)  \ \dot=\ -\log P(\exists p^*: |p^*-p| < \delta).
\end{equation}
 
The   {\it Laplace transform (L.t.) } of the total excess demand $Z(p)$
is defined as the function
\begin{equation}
\label{equ232a}
\Lambda(\alpha;p)\ \dot=\ Ee^{\alpha \cdot Z(p)}\ =\ \int e^{\alpha
\cdot
Z(\omega;p)}P(d\omega),\ \alpha \in \R^l.\end{equation}
 The {\it (macroeconomic) partition function}
$\Lambda(p)$
is defined as its minimum:
$$\Lambda(p)\ \dot=\ \min\limits_{\alpha \in 
\R^l} \Lambda(\alpha;p).$$
\end{definition}
%\medskip

The logarithm of the L.t. of a random variable (the so-called {\it
cumulant generating function})  is known to be a convex function, see
e.g. \cite{bil}: p. 148. 
It follows, in particular, that if there exists a parameter
$\alpha = \alpha(p) \in \R^l$ such that
\begin{equation}
\label{equ233}
{\partial \log \Lambda \over \partial \alpha}(\alpha(p);p)=0,
\end{equation}
then
necessarily \begin{equation}
\label{equ233a}
\Lambda(p)=\Lambda(\alpha(p);p).\end{equation}
Due to the thermodynamic analogy with the concept of inverse
temperature, the variable $\alpha(p)$  is called the {\it conjugate variable}
(associated with  price $p$).
(Later we show that the existence of a conjugate variable is
equivalent to $p$ being a so-called {\it possible equilibrium price}, see Theorem \ref{Theorem 1.} in  Section \ref{sub42}.)
\bigskip

According  to  the Second Law  of stochastic equilibrium economics, the economic entropy equals the negative of the logarithm of the partition function:

\medskip
{\bf The Second Law of Stochastic Equilibrium Economics:}
\begin{equation}
\label{equ234}
I(p)=-\log \Lambda(p).
\end{equation}

\bigskip

\textsc{Remarks:}
(i) The observation (\ref{equ231}) is the economic analog of the observation (\ref{equ210}) of temperature in thermodynamics.
\medskip

(ii) The Second Law is obtained at the  limit $n = \infty,\ \delta = 0,$ from
a   theorem of large deviations concerning  the random equilibrium prices. According to the TLD
$${\cal I}(\exists p^*: |p^*-p| < \delta)+\log \Lambda(p) =
\varepsilon(n,\delta)n,$$
where $\varepsilon(n,\delta) \to 0$ as  $n \to
\infty$ and $\delta \to 0$ , see Theorem \ref{Theorem 2.} in Section \ref{sub43}.
\medskip

(iii) Entropy is a "measure of rareness" for the possible
values $p$ for the random equilibrium prices. Namely, a priori
unexpected "rare" values  for the r.e.p. have large
entropy whereas a priori expected "common" values have small entropy.

Thus the second law of  can be interpreted as providing an
information theoretic  measure of goodness for the a priori
equilibrium model:

Namely,    if the a priori equilibrium model is "good",  then
due to the LLN, the observed value $p$ for the random equilibrium price
$p^*$ is near to its a priori expected value $p^*_e$   and, therefore, has  high a priori
probability. Thus its entropy $I(p)$, which  by definition equals its
information content,  is small.
On the other hand,  the  realization of an a priori rare value for the
random equilibrium price  has large entropy, indicating the "badness" of
the a priori model.

\medskip
(iv) A {\it partial observation} of the r.e.p.  means the observation of
the r.e.p. $p^*$  in some subset  of the price simplex:
$$\exists p^* \in B,\ \hbox{ where}\ B \subset S^l.$$ 
This is the case, e.g., if the prices of some subset of the commodities  are observed, see \cite{num4}: Ex. 3.1. Also, if in a dynamical finance market the prices of the assets are observed only up to some finite time, then this represents a partial observation (of the whole price process), see \cite{num4}:  Ex. 3.2.

According to the {\it principle of minimum entropy}, the entropy of
a partial observation of the equilibrium price is
equal to the entropy of the {\it entropy minimizing price} in the
observation set: $$I(B) \doteq {\cal I}(\exists p^* \in B) = I(p^*_B),$$  where
$$p^*_B \ \dot=\ \hbox{argmin} \{I(p):\ p \in B\}.$$ 
This is a {\it large deviation theorem concerning partial observations}, see \cite{num4}: Theorem 3.4.

\medskip

(v) By an  observation of a {\it random
composite equilibrium}    at the price-variable pair
 $(p,x) \in S^l \times \R^d$  we mean the
occurrence of the event 
\begin{equation}
\label{equ236}
\{\omega \in \Omega:\ \hbox{there exists a r.e.p.}\ p^*(\omega) \ \hbox{such
that}\ |p^*(\omega)-p| < \delta \ \hbox{and} \  |x^*(\omega)-x|
< \delta\},\end{equation}
where $x^*(\omega) \dot= n^{-1}X(\omega;p^*(\omega))$
denotes the mean of the macroeconomic random variable
$X(\omega; p)$ at the equilibrium.
The observation (\ref{equ236}) will be written shortly as \begin{equation}
\label{equ236a}
\exists p^*:  \ |p^*-p| < \delta,  |x^*-x| < \delta.
\end{equation} 
Also here $\delta >0$ is an "infinitesimally small" constant.%*) MIHIN TAMA TAHTI MAHTAA VIITATA? JATETAAN VARMAAN POIS!

The {\it  (bivariate) entropy} $I(p,x)$  associated with the observation
of a r.c.e. at $(p,x)$  is defined as the   information content in
this observation:
\begin{eqnarray}
I(p,x) & \dot=  & {\cal I}(\exists p^*: |p^*-p| < \delta,
|x^*-x| < \delta) \nonumber \\
&\dot=&-\log P(\exists p^*: |p^*-p| < \delta,  |x^*-x| < \delta). \nonumber
\end{eqnarray}

The {\it bivariate partition function} $\Lambda(p,x)$ is defined by the formula
 $$\Lambda(p,x)\ \dot=\ \Lambda(\alpha(p,x),\beta(p,x);p),$$ where
$$\Lambda(\alpha,\beta;p)\ \dot=\ Ee^{\alpha \cdot Z(p)+\beta \cdot
X(p)}= \int e^{\alpha \cdot Z(\omega;p)+\beta \cdot
X(\omega;p)}P(d\omega),\ \alpha \in R^l, \beta \in R^d,$$ denotes  the
{\it bivariate Laplace transform},  and the {\it bivariate conjugate
variables} $\alpha (p,x)$ and $\beta(p,x)$  are the solutions of the  equations
\begin{eqnarray}
{\partial \over \partial \alpha}\log
\Lambda(\alpha(p,x),\beta(p,x);p) &=& 0, \nonumber \\
{\partial \over \partial \beta}\log
\Lambda(\alpha(p,x),\beta(p,x);p) &=& nx,\label{equ237}
\end{eqnarray} cf. \cite{num4}: Section 4.3.

According to a {\it generalized second law}, the entropy of a bivariate composite equilibrium can be expressed in terms of the bivariate partition function $\Lambda (p,x)$ and the conjugate variable
$\beta (p,x)$:
\begin{equation}
\label{equ238}
I(p,x) =-\log \Lambda(p,x)+ n\beta(p,x) \cdot x,\end{equation}
cf. \cite{num4}: Section 4.3.
\medskip

In the case of Example (iii) in Section \ref{sub22},  the observation of a random composite equilibrium
$(p^*,n^*)$  at $(p,x)$ has the meaning of a simultaneous observation of
the r.e.p. $p^*$ at $p$ and  of the  proportion  $n^* \ \dot=\ n^{-1}N(p^*)$  of non-surviving agents at $x$.
The bivariate  entropy $I(p,x)$ is the information content of this
observation.

\subsection{Ideal random economies}

\label{sub24}

We now illustrate the general results with the aid of a special
class of simple random economies, which, due to their analogy with the
classical {\it ideal gas} of statistical mechanics, are called
 {\it ideal random economies}:

\medskip

We assume that the {\it individual excess demand} $\zeta_i(p)$ by agent
$i$ is obtained with the aid of  a   deterministic {\it structure
function} $z(\theta_i;p)$   of a  random
parameter $\theta_i$ (the {\it economic characteristics} of agent $i$)
%$\theta_i \in R^m$ (for some $m \geq 1$)
and of  price $p$:
$$\zeta_i(p) = z (\theta_i;p). $$
The  economic characteristics $\theta_i$  are supposed to form a sequence of
$\R^m$-valued (for some $m \geq 1$)  independent and identically distributed (i.i.d.) random
variables.

Note that, since statistical independence
is preserved under deterministic transformations, it
follows  that, for each  price $p$, the  individual excess demands
 are i.i.d., too. Thus in an ideal economy the  total
excess demand is the sum of i.i.d. random variables:
\begin{equation}
\label{equ241}
Z(\omega;p) =  \sum\limits^n_{i=1}z(\theta_i;p).
\end{equation}
The macroeconomic configuration $\omega$ is now defined as the vector of the
individual characteristics: $ \omega \ \dot=\
(\theta_1,...,\theta_n) \in \Omega \ \dot=\ \R^{mn}.$
\medskip

Let $f(\theta)$  denote the common  probability distribution  function (p.d.f.)
of the economic characteristics, i.e.,
\begin{equation}
\label{equ242}
P(\theta_i \in A) = \int\limits_A f(\theta)d\theta \ \ \hbox{for}
\ i=1,2,..., \ A \subset \R^m.\end{equation}
We call $f(\theta)$  the {\it a priori microeconomic p.d.f.}
It follows that the a priori
macroeconomic probability law  $P(d\omega)$ is  the product  probability
law, under which the economic
characteristics $\theta_i$ are i.i.d. $f(\theta)$-distributed random
variables, viz.
\begin{equation}
\label{equ242a}
P(d\omega)=f(\theta_1) \cdots f(\theta_n)d\theta_1 \cdots d\theta_n.\end{equation}

Let
$$\mu(p) \ \dot =\ \ E\zeta_i (p) = \int z(\theta;p) f(\theta)d\theta$$
denote  the {\it expected individual excess demand}.
%Due to the LLN for i.i.d. r.v.'s, it is  equal to  the limit of the  mean total excess demands: $$\mu(p)  = \lim\limits_{n \to \infty} n^{-1}Z(p).$$
Since $$EZ(p)=n \mu(p),$$  the expected equilibrium prices are also zeros of the
expected  individual excess demand: \begin{equation}
\label{equ243}
\mu(p^*_e)=0.\end{equation}
%\medskip

Due to the independence of the individual total 
excess demands, the Laplace
transform  of the total excess demand in an ideal random economy is equal to
the $n$'th power of the   L.t.  of the individual excess demand:
\begin{eqnarray}
 \Lambda(\alpha;p) &\dot=& E e^{\alpha \cdot Z(p)} 
\nonumber 
\\
&=& E  e^{\alpha \cdot  \sum\limits^n_{i=1}\zeta_i(p)} \nonumber 
\\
&=& E e^{\alpha \cdot \zeta_1(p)} \cdots Ee^{\alpha \cdot \zeta_n(p)} \nonumber 
\\
&= &\lambda(\alpha;p)^n, \nonumber 
\end{eqnarray}
where $$\lambda(\alpha;p)  \dot=  E e^{\alpha \cdot
\zeta_i(p)} = \int e^{\alpha \cdot
z(\theta;p)}f(\theta)d\theta.$$

It follows that the partition function is the $n$'th power of the
  individual partition function $\lambda(p)$:
\begin{equation}
\label{equ244}
\Lambda(p)= \lambda(p)^n,\end{equation}
where
\begin{equation}
\label{equ245}
\lambda(p)  \dot =  \min\limits_{\alpha}  \lambda (\alpha;p) 
 = \lambda (\alpha(p);p),\end{equation}
and the conjugate variable $\alpha(p)$ satisfies the equation
\begin{equation}
\label{equ246}
{\partial \over \partial
\alpha}\log \lambda(\alpha(p);p)=  0,\end{equation}
cf. (\ref{equ233}).

In view of (\ref{equ234}), the  Second Law for an ideal economy obtains the form \begin{equation} \label{equ247} I(p)=-n\log \lambda(p).\end{equation} 
\medskip 

\textsc{ Examples: }
%\medskip
(i) In an  ideal  random Cobb-Douglas  economy the economic characteristics
$$\theta_i \ \dot =\ (a_i,e_i)\in S^l \times R^{l+1}$$  form an i.i.d. sequence of
random variables.
In view of formula (\ref{equ221}), the  structure function is
\begin{equation}
\label{equ248}
 z(\theta;p)=z(a,e;p)\ \dot=\ ({a^j \over p^j}p \cdot e - e^j;\
j=1,...,l), \ \theta \dot= (a,e) \in S^l \times R^{l+1}.\end{equation}

It follows that the  expected individual excess demand on the
commodity $j$ is given by
\begin{equation}
\label{equ249}
\mu(p) =  ((p^j)^{-1}\sum\limits_{k=1}^{l+1} \mu_{a;e}^{jk} p^k -
\mu_e^j;\ j=1,...,l),\end{equation} where
$$\mu_{a;e}^{jk}  \dot = E(a^j_ie^k_i) = \int \int a^je^k
f(a,e)dade, $$
$$\mu_e^k \dot= Ee^k_i = \sum\limits_{j=1}^{l+1} \mu_{a;e}^{jk},$$
and $f(\theta)=f(a,e)$ denotes the microeconomic p.d.f..

In view of equations (\ref{equ243}) and (\ref{equ249}), the  expected equilibrium price is given by
$$p_e^*=\,\biggl ({(w_e^*)^j \over \mu^j_e};\ j=1,...,l\biggr ),$$ where
$w_e^*$ is a left eigenvector of the stochastic matrix 
$$A_e \ \dot =\,
\biggl ({\mu_{a;e}^{kj} \over \mu^j_e};\ j,k=1,...,l+1 \biggr),$$  subject to the
normalization  $p^*_e \in S^l$, cf. (\ref{equ221a}) and (\ref{equ221b}).
If  $A_e$ is irreducible, then
there  is only one unique expected equilibrium price $p^*_e$.

The Laplace transform of the individual excess demand
in  an ideal random CD economy  is given by
\begin{eqnarray}
\lambda(\alpha;p) &=&
\int\limits_{S^l} \int\limits_{R^{l+1}} e^{\alpha \cdot z(a,e;p)
}f(a,e)dade \nonumber \\
&=& \int\limits_{S^l} \int\limits_{R^{l+1}}
e^{\sum\limits_{j=1}^l \alpha^j  ( (p^j)^{-1}a^j p \cdot e - e^j)} f(a,e)dade.\nonumber
\end{eqnarray}

The conjugate parameter $\alpha(p) \in R^{l}$ is the
solution of equation 
(\ref{equ246})
,
viz., presently,  of  the system $$\int\limits_{S^l}
\int\limits_{R^{l+1}} ((p^j)^{-1}a^j p \cdot e - e^j) e^{\Sigma_{j=1}^l\alpha^j(p) ( (p^j)^{-1}a^j p \cdot e
- e^j)}f(a,e)dade = 0,\ j=1,...,l.$$
For the individual partition function we obtain
the formula
\begin{eqnarray}
\lambda(p) &=& \lambda(\alpha(p);p) \nonumber \\
&=& \int\limits_{S^l} \int\limits_{R^{l+1}}
e^{\sum\limits_{j=1}^l \alpha^j(p)  ( (p^j)^{-1}a^j p \cdot e - e^j)}
f(a,e)dade,\nonumber
\end{eqnarray} cf. (\ref{equ245}). In view of (\ref{equ247}),  the Second Law for an ideal random CD economy obtains the form
$$I(p) = - n\log  \int\limits_{S^l} \int\limits_{R^{l+1}}
e^{\sum\limits_{j=1}^l \alpha^j(p)  ( (p^j)^{-1}a^j p \cdot e - e^j)}
f(a,e)dade.$$
\medskip

(ii) 
An {\it ideal  random financial market} is formed by
$n$ statistically independent and identical financial agents.
The natural choice for the economic characteristics $\theta_i$ of agent
$i$ is the
$m$-dimensional ($m=l^2+2l+2$) vector comprising his risk parameter
$a_i \in R$, the vector of his subjective expectations $\mu_{\psi;i} \in
R^l$ and covariances $\Sigma_{\psi;i} \in R^{l \times l},$ and his
initial endowment $e_i \in R^{l+1}$:
$$\theta_i \ \dot =\ (a_i,\mu_{\psi;i},\Sigma_{\psi;i},e_i),\
i=1,...,n.$$
The structure function of the individual excess demand is (cf. \cite{blo})
\begin{eqnarray}
z(\theta;p)
 &=& z(a,\mu_{\psi},\Sigma_{\psi},e) \nonumber \\
  &\dot=& (a\Sigma_{\psi})^{-1} \mu 
-{p^T(
a \Sigma_{\psi})^{-1} \mu 
-p^Te
\over
p^T(a\Sigma_{\psi})^{-1}p}(a \Sigma_{\psi})^{-1}p.\nonumber\end{eqnarray}

%We will not pursue further with this example but plan to investigate it
%in a later study \cite{num5}.

\medskip

(iii) Consider the {\it survival model} as described in Section \ref{sub22} (Example (iii) therein).

Due to the LLN, under appropriate regularity conditions,  the proportion $n^* =n^{-1} N(p^*)$
 of non-surviving agents at equilibrium price $p^*$ is a priori near to   the  probability of
non-survival of a randomly chosen agent at the expected equilibrium
price:  $$n^* \approx n^*_e \ \dot=\ P(p^*_e \cdot e < \underline
w(p^*_e)).$$
However, again,  due to the imperfectness of the a priori model, the actually
realized  proportion $n^*$ may well represent a large deviation within
this model.
The information content of the  simultaneous
observation of the r.e.p. 
at price $p$
and the proportion  of non-surviving agents at $x$ 
is given by the generalized Second Law (\ref{equ238}). Due to the postulated statistical
independence of the agents, the generalized Second Law  obtains now the form (cf. (\ref{equ247}))
$$I(p,x)= -n \log \lambda(p,x)+ n \beta(p,x) \cdot x,$$
where
$\lambda(p,x)$ and $\beta(p,x)$  denote the {\it individual partition
function} and the {\it conjugate parameter} defined as the solutions of the equations (\ref{equ237}).

%We will neither pursue further with this example but plan to investigate it
%in a later study \cite{num6}.

\subsection{The Second Law and the  Central Limit Theorem}

\label{sub25}

According to the Law of Large Numbers, a r.e.p. is a priori
"near to" its expected value:
$$p^*  \to p^*_e \ \hbox{ as}\ n \to
\infty,$$ (\cite{bhamaj1}, \cite{num1}).

The  Central Limit Theorem for the r.e.p.'s (\cite{bhamaj1})
characterizes the
"small deviations" of the r.e.p. $p^*$ from its  expected value $p^*_e$ as
asymptotically normally distributed. Namely, under appropriate
regularity conditions
 $$\sqrt n(p^* - p^*_e) \to {\cal N}(0,\Sigma) \ \hbox{ as}\ n \to
 \infty,$$   where
${\cal N}(0,\Sigma)$ denotes a
multinormal random vector with zero mean and covariance $\Sigma $.
Thus
\begin{equation}
\label{equ251}
p^* \approx  {\cal N}(p^*_e,n^{-1}\Sigma)\ \ \hbox{for large}\
n
\end{equation}
so that  the  standard deviation of the distribution of the r.e.p. $p^*$
itself is of
the  asymptotically small order ${1 \over \sqrt n}$.   This means that
the CLT describes  the random fluctuations at the
{\it "mesoeconomic  intermediate
scale"} ${1   \over \sqrt n}$   between the {\it "micro-"} and {\it "macroeconomic scales"} ${1 \over n}$
and $1$. 

It follows that,  if the observed value  $p$ for the r.e.p. $p^*$ happens to fall within a distance of the  order  ${1   \over \sqrt n}$ from its  expected value $p^*_e$, then, due to  (\ref{equ251}), the  probability of this observation   can be approximated with the aid of the CLT:
\begin{eqnarray}
P(\exists p^*: |p^*-  p| < \delta)
&=& P(\exists p^*: |\sqrt n p^* -\sqrt n p|  < \sqrt n \delta) \nonumber \\
&\approx & C n^{{l \over 2}}
\delta^l
 e^{- {n \over 2} ( p - p^*_e)^T
\Sigma^{-1}( p - p^*_e)}, \nonumber
\end{eqnarray}
where $C$ is a constant and, again,  the
"tolerance" $\delta > 0$ is supposed to be small. Furthermore,
since
${\log (C n^{{l \over 2}} \delta^l) \over n} \to 0$  as $n \to \infty,$
this probability has the exponential order
$$e^{- {n \over 2} ( p - p^*_e)^T \Sigma^{-1}( p - p^*_e)}.$$ Therefore,  at a distance of the order   ${1   \over \sqrt n}$   from $p^*_e$, the entropy is approximately
$$I(p)\ \dot=\ - \log P(\exists p^*: |p^*-  p| < \delta) \approx {n \over 2} ( p - p^*_e)^T \Sigma^{-1}( p - p^*_e).$$

This CLT-based approximation is  consistent with the Second Law (as it ought to be). Namely, $$ -\log \Lambda(p) \approx {n \over 2} ( p - p^*_e)^T \Sigma^{-1}( p - p^*_e),$$   when $p$ is close to $p^*_e $ (cf. \cite{num2}: formula (3.2)] and \cite{bhamaj1}: Theorem 4.1(iii)]).

Outside its (narrow) region of validity  this CLT-based approximation  of the Second Law is no longer valid and its use is therefore not mathematically justified.

\medskip

\textsc{ Remark.}
{\it Mesoscopic scale} in  statistical mechanics refers to the small
Gaussian random fluctuations of the thermodynamic equilibrium.

\section{Gibbs Conditioning Principle}

\label{sec3}

\subsection{Gibbs Conditioning Principle in thermodynamics}

\label{sub31}

Let $\beta > 0$ be an arbitrary  fixed inverse temperature.

In view of the definition 
(\ref{equ212})
of the thermodynamic partition function $\Lambda(\beta)$,
\begin{equation}
\label{equ311}
P(d \omega|\beta)\ \dot=\ \Lambda(\beta)^{-1} e^{-\beta
U(\omega)}d \omega \end{equation}
is a probability law on the
thermodynamic ensemble $\Omega$. It is called the
{\it canonical probability law}  (at $\beta$).
The corresponding probability distribution function (p.d.f.)
\begin{equation}
\label{equ311a}
p(\omega|\beta)\ \dot=\ \Lambda(\beta)^{-1} e^{-\beta
U(\omega)}\end{equation}
 is called the {\it canonical probability
distribution function}.

\medskip
The energy $U=(U(\omega); \omega \in \Omega)$ can be regarded as a  random variable under the canonical probability law $P(d\omega|\beta)$, and the internal energy $E(\beta)$ can be interpreted as its  expectation.
Namely, in view of (\ref{equ212}) and (\ref{equ213}), we have
\begin{eqnarray}
E(U|\beta) & \dot= &
\int  U(\omega)P(d\omega| \beta) 
\nonumber \\
& = &  \Lambda(\beta)^{-1}   \int U(\omega) e^{-\beta  U(\omega)}d\omega 
\nonumber \\
& = &  -{d \over d\beta} \log   \int e^{-\beta  U(\omega)}d\omega 
\nonumber \\
& = & E(\beta).\label{equ312}
\end{eqnarray}

According to the {\it Gibbs Conditioning Principle (GCP)}, a thermodynamic system is governed by the canonical 
probability law at the measured
 temperature.
(Recall that the measurement of the temperature $T={1 \over \beta}$ means that the energy
$U(\omega)$ is in the neighborhood of the associated internal energy
$E(\beta)$, see (\ref{equ210}).)

\medskip

\textsc{Example: The ideal gas.}
%\medskip
In view of 
(\ref{equ217a}) and 
(\ref{equ218}), the canonical probability law
for the ideal gas has the following product form:
$$P(d\omega| \beta) = \Lambda (\beta)^{-1} e^{-\beta \sum\limits_{i=1}^n
u(\theta_i)} d\omega = f(\theta_1|\beta) \cdots  f(\theta_n|\beta)
d\theta_1 \cdots d\theta_n,$$
where
\begin{equation}
\label{equ313} 
f(\theta|\beta) = \lambda(\beta)^{-1} e^{-\beta u(\theta)}
= (2\pi m)^{-{3 \over 2}} \beta^{{3 \over 2}}e^{-{\beta |\theta|^2
\over 2m}},\ \theta \in R^3,\end{equation}
is a Gaussian probability distribution function on $R^3$, called the {\it microcanonical p.d.f.}

Thus, according to the Gibbs Conditioning Principle, the
momenta  $\theta_i$ of the particles  in the ideal gas are i.i.d.
random variables obeying the  Gaussian microcanonical p.d.f.

\subsection{Gibbs Conditioning Principle in stochastic equilibrium economics}

\label{sub32}

The Gibbs conditioning principle has an analogy in stochastic equilibrium economics:

Suppose that we observe the random equilibrium price $p^*$ at price
$p$ (see
\ref{equ231}
). Due to the Law of Large Numbers, the observed value $p$ is
necessarily equal (or at least "near") to the expected equilibrium price
under the unknown
"true" a posteriori macroeconomic probability law $P(d\omega|\exists
p^*:|p^*-p|< \delta)$. Thus, at the ideal limit $n=\infty,\ \delta = 0$, 
we ought to have
$$E(Z(p)|\exists p^*:|p^*-p|< \delta) \dot= \int Z(\omega;p)
P(d\omega|
\exists p^*:|p^*-p|< \delta)
 =0.$$ 

This implies that  the  a posteriori
macroeconomic probability distribution function $g(\omega|
\exists p^*:|p^*-p|< \delta)
$,  defined
as the  density of the a posteriori macroeconomic probability law
w.r.t. the apriori law,
$$g(\omega|
\exists p^*:|p^*-p|< \delta)P(d\omega)
\ \dot=\ P(d\omega|\exists p^*:|p^*-p|< \delta)$$
ought to satisfy the relation
\begin{eqnarray}
\label{equ321}
\int Z(\omega;p) g(\omega|
\exists p^*:|p^*-p|< \delta)
P(d\omega)=0.
\end{eqnarray}

Although condition (\ref{equ321}) is necessary for the a posteriori p.d.f., it alone is not sufficient for its unique characterization. The economic analog of the Gibbs conditioning principle will characterize the so-called  {\it canonical macroeconomic p.d.f.} as the a posteriori p.d.f. As it should, the canonical macroeconomic p.d.f. satisfies (\ref{equ321}), see (\ref{equ324}).

\medskip

Suggested by condition (\ref{equ321})
we give the following  definition:
\medskip

%\textsc{ Definition 2.} 
\begin{definition}\label{def2}
A price $p$ is called an {\it  (a priori) possible
equilibrium price (p.e.p.)} if there is a strictly positive probability distribution
function $g(\omega;p) > 0,\ \int g(\omega;p)P(d\omega)=1$, such that $p$
is an expected
equilibrium price under the transformed macroeconomic probability law
$P_g(d\omega;p)\ \dot=\ g(\omega;p)P(d\omega)$, viz.,
\begin{equation}
\label{equ322}
E_gZ(p) \ \dot=\ \int Z(\omega;p) g(\omega;p)P(d\omega)=0.
\end{equation}
\end{definition}

A probability distribution function
$g(\omega;p)$,
which satisfies (\ref{equ322}), can be regarded as a  candidate for the
a posteriori macroeconomic p.d.f. $g(\omega|\exists p^*: |p^*-p|< \delta)$.
\medskip

With any price $p$, for which  the conjugate parameter $\alpha(p)$
satisfying equation (\ref{equ233})
exists, we can associate  a p.d.f.;  due to the thermodynamic analogy, it is called the  {\it canonical
macroeconomic p.d.f.}:
\begin{equation}
\label{equ322a}
g(\omega|p)\ \dot=\ \Lambda(p)^{-1} e^{\alpha(p) \cdot Z(\omega;p)}
\end{equation}

We prove later in Theorem \ref{Theorem 1.}
that the existence of a conjugate parameter $\alpha(p)$ is equivalent to $p$ being a possible equilibrium price.
The notation for the canonical macroeconomic p.d.f anticipates its role as the a posteriori macroeconomic p.d.f.

The probability law
\begin{equation}
\label{equ323}
P(d\omega|p)\ \dot=\ g(\omega|p)P(d\omega) = \Lambda(p)^{-1} e^{\alpha(p) \cdot
Z(\omega;p)}P(d\omega)
\end{equation} is called the {\it canonical (macroeconomic) probability law}.

In analogy with the thermodynamic formula (\ref{equ312}), price $p$ is an expected equilibrium price under
the associated canonical probability law:

Namely, in view of (\ref{equ232a}), (\ref{equ233}) and
(\ref{equ233a}), we have \begin{eqnarray} E(Z(p)|p)  &\dot=&
\int Z(\omega;p)g(\omega|p) P(d \omega) \nonumber \\
& = &\Lambda(p)^{-1} \int Z(\omega;p)e^{\alpha(p) \cdot
Z(\omega;p)}P(d\omega)   \nonumber \\ &=& \Lambda(\alpha(p);p)^{-1} {\partial \over
\partial \alpha}\Lambda(\alpha(p);p) \nonumber \\
&=& {\partial \over \partial \alpha}\log \Lambda(\alpha(p);p) \nonumber \\
&=& 0. \label{equ324}
\end{eqnarray}
This means that
 the canonical p.d.f. is a candidate for the a posteriori macroeconomic p.d.f.

\medskip
Now, in fact, according to the economic analog of the 
Gibbs Conditioning Principle, 
conditionally on the observation of the random equilibrium price at a possible equilibrium price, the ensuing 
a posteriori macroeconomic
probability distribution function is given by  the canonical  macroeconomic
p.d.f.:
\medskip

{\bf Gibbs Conditioning Principle in stochastic equilibrium economics:}
\begin{equation}
\label{equ324a}
g(\omega|\exists p^*:|p^*-p|< \delta)= g(\omega|p).\end{equation}
\bigskip

\textsc{Remarks:}
%\medskip
(i) Of course, equivalently with (\ref{equ324a}), the a posteriori macroeconomic probability law is given by  the canonical macroeconomic probability law:
$$P(d\omega|\exists p^*:|p^*-p|< \delta)= P(d\omega|p).$$
In what follows we shall often formulate  GCP in terms of probability laws rather than probability distribution functions.

\medskip
(ii)
Geometrically, the definition of the possible equilibrium price  means that
$0$ belongs to the  topological interior of the convex hull of the support of the distribution of the total excess demand, see Theorem  \ref{Theorem 1.} in Section \ref{sub42}.
\medskip

(iii) As a  theorem, GCP is a conditional Law of Large
Numbers  concerning  macroeconomic random variables (Theorem \ref{Theorem 3.} in Section \ref{sub52}). According to the conditional LLN,
conditionally on the observation of the
random equilibrium price,  macroeconomic random variables are centered at their  canonical
expectations.

\medskip
(iv)  According to the Principle of Minimum Entropy (PME),
 conditionally on
a partial observation of the random  equilibrium price (see Remark (iv) in \ref{sub23}),
the a posteriori r.e.p. is equal to
the entropy minimizing price which is compatible with the observation:
$$ \exists p^* \in B \Rightarrow p^* = p^*_B,$$ cf. 
\cite{num4}
: Theorem 3.7.
Combining this with GCP, it follows that, conditionally on a partial
observation of the equilibrium price,  the a posteriori macroeconomic
probability law is
given by the canonical law associated with the entropy-minimizing price:
$$P(d\omega|\exists p^* \in B) = P(d\omega|p^*_B),$$ cf. \cite{num4}: Theorem
4.7.

%In the forthcoming study \cite{num5} we apply this principle in the
%prediction of asset prices in a dynamical asset market.
\medskip

(v) In analogy with the characterization of a possible equilibrium price  as a zero of the
derivative of the c.g.f. of the total excess demand, one may characterize a {\it possible composite equilibrium (p.c.e.)} as a
pair $(p,x) \in S^l \times R^d$ such that it is possible to relate
the conjugate variables $\alpha(p,x)$ and $\beta(p,x)$ to each other via
equations  (\ref{equ237}).

The {\it canonical macroeconomic probability law}
associated with a p.c.e. $(p,x)$ is defined by
\begin{equation}
\label{equ325}
P(d\omega|p,x)\ \dot=\ \Lambda(p,x)^{-1} e^{\alpha(p,x) \cdot
Z(\omega;p)+\beta(p,x) \cdot
X(\omega;p)}P(d\omega).\end{equation}

According to a generalized Gibbs Conditioning Principle, conditionally on the observation of  a random composite equilibrium at the price-variable pair $(p,x)$, the governing a posteriori macroeconomic probability law is
given by the canonical macroeconomic probability law (\ref{equ325}) associated with the observation:
$$P(d\omega|\exists p^*: |p^*-p| < \delta,|x^*-x| <
\delta) =P(d\omega|p,x).$$
%see \cite{num6}.

\subsection{Gibbs Conditioning Principle for  ideal random economies}
\label{sub33}

Consider an ideal random economy as described in Section \ref{sub24}. In view of (\ref{equ241}), (\ref{equ242a}), (\ref{equ244}) and (\ref{equ323}), the canonical macroeconomic probability law has the product form
\begin{eqnarray}
\label{equ331}
P(d\omega|p) &=& \Lambda(p)^{-1} e^{\sum\limits_{i=1}^n
\alpha(p) \cdot z(\theta_i;p)}f(\theta_1) \cdots f(\theta_n)d\theta_1
\cdots d\theta_n  \nonumber \\ &=& f(\theta_1|p) \cdots f(\theta_n|p)d\theta_1
\cdots d\theta_n,
\end{eqnarray}
where 
\begin{equation}
\label{equ331a}
f(\theta|p)\ \dot = \ \lambda(p)^{-1}e^{\alpha(p) \cdot
z(\theta;p)}f(\theta) \end{equation} 
is a probability distribution function
on the set $\Theta$ of economic characteristics. It 
is called the {\it
canonical microeconomic p.d.f.} (associated with price $p$).
Thus, under the canonical macroeconomic probability law,  the
economic characteristics are i.i.d. random variables 
obeying the canonical microeconomic 
p.d.f.

From applying the GCP (\ref{equ324a}) now follows
that,  
in an ideal random
economy and conditionally on observing a r.e.p. at $p$, the economy is
still ideal, i.e., the economic characteristics are i.i.d. and 
the a posteriori microeconomic p.d.f. is given by the
canonical microeconomic p.d.f.:
\begin{equation}
\label{equ332}
f(\theta|\exists p^*:|p^*-p| < \delta)= f(\theta|p).
\end{equation}
\medskip

\textsc{Examples:}
%\begin{example}
(i) Consider again an ideal random Cobb-Douglas economy. According to  GCP, conditionally on the observation of an 
equilibrium
price, the a posteriori economy is still an ideal Cobb-Douglas economy having
the canonical microeconomic p.d.f. as the a posteriori microeconomic
p.d.f.:
\begin{eqnarray}
\label{equ333}
f(a,e|\exists p^*:|p^*-p|< \delta) &=&  f(a,e|p)
\nonumber \\ &\dot=& \lambda(p)^{-1}e^{\alpha(p) \cdot z(a,e;p)}f(a,e)
\nonumber \\ &=& \lambda(p)^{-1}e^{\Sigma_{j=1}^l\alpha^j(p) (
(p^j)^{-1}a^j p \cdot e - e^j)}f(a,e). \end{eqnarray}
cf. (\ref{equ331a}) and (\ref{equ248}).
%\end{example}
\medskip

%\begin{example}
(ii) Consider again the survival model of Example (iii) in Section \ref{sub22} and assume that the underlying Cobb-Douglas economy is ideal. Suppose that the random equilibrium price $p^*$
is observed at price $p$ and the proportion $x^*$
of non-surviving agents at $x$ (cf. Remark (iv) in Section \ref{sub23}).
The ensuing a posteriori microeconomic p.d.f. is now given by 
the associated canonical microeconomic p.d.f. $f(a,e|p,x)$:
$$f(a,e|\exists p^*: |p^*-  p|< \delta, |x^*-  x|< \delta) =
f(a,e|p,x).$$

If only the proportion $x^*$ of non-surviving agents is  observed, i.e.,
we have only a partial
observation of the random composite equilibrium $(p^*,x^*)$, then,
according to the PME and GCP,  the a posteriori microeconomic p.d.f. is
given by
$$f(a,e| |x^*-  x|< \delta) = f(a,e|x)\ \dot=\ f(a,e|p(x),x),$$
where
$p(x)$ denotes the price that minimizes the bivariate entropy (\ref{equ238}) over
the price variable:
$$I(p(x),x) = \min\limits_p I(p,x)= \min\limits_p (-n \log \lambda(p,x)+
n \beta(p,x) \cdot x).$$ 
%\end{example}%see\cite{num6}.
%\bigskip 

\begin{Large}

\bigskip

\noindent PART II: EXACT RESULTS

\end{Large}

\section{Theorems of Large Deviations}

\label{sec4}

\subsection{The Second Law of thermodynamics as a Theorem of Large Deviations}

\label{sub41}

Let $\beta > 0$ be an arbitrary fixed inverse temperature.

Recall from Section \ref{sub31} that the energy $U=\{U(\omega);\omega \in \Omega\}$ of a thermodynamic system can be regarded as a  random variable on the ensemble $\Omega$ under the canonical law $P(d\omega|\beta)$. Also recall  that then the internal energy $E(\beta)$ equals the expectation of $U$.

Suppose that the energy satisfies the (weak) LLN under the canonical probability law $P(d\omega|\beta)$; namely,
\begin{equation}
\label{equ411}
\lim\limits_{n \to \infty}P(n^{-1}|U-E(\beta)| <  \varepsilon|\beta)=1\ \hbox{for all}\
\varepsilon > 0.
\end{equation}

\medskip
Under this hypothesis the thermodynamic Second Law can be formulated as a Theorem of Large Deviations (TLD) concerning the energy.

To this end, let $\delta > 0$ be an arbitrary fixed constant. Then one can
prove that 
\begin{equation}
\label{equ412}
|\log Vol
(|U-E(\beta)| <  n\delta)
-(\log \Lambda(\beta)+ \beta E(\beta))| < n\delta 
\ \hbox{eventually}.
\end{equation}
The phrase "eventually" means the same as "for all sufficiently large $n$".
Clearly this implies also the following:
\begin{equation}
\label{equ413}
\lim\limits_{\delta \to 0}\limsup\limits_{n \to \infty} |n^{-1}[\log Vol
(|U-E(\beta)| <  n\delta)-(\log \Lambda(\beta)+ \beta E(\beta))]|
=0.\end{equation}
The (integral form of the) thermodynamic Second Law (\ref{equ217})  is obtained from (\ref{equ413})  at the ideal limit $n =\infty,\ \delta =0.$

\medskip

To prove (\ref{equ412}), note first that, due to (\ref{equ411}),  
$$\gamma_n\ \dot=\ \log P(|U-E(\beta)| <
{n\delta \over 2 \beta}) \to 0 \ \hbox{as}\ n \to \infty.$$
On the other hand, in view of the definition (\ref{equ311}) of the canonical probability law, we can also write
\begin{equation}
\label{equ414}
  \gamma_n = \log  \int\limits_{\{|U-E(\beta)| <
{n\delta \over 2 \beta} 
|\beta \}} e^{- \beta U(\omega)} d\omega - \log
\Lambda(\beta).
\end{equation}

Since clearly, \begin{equation}
\label{equ415}
|\log
\int\limits_{\{|U-E(\beta)| < {n\delta \over 2 \beta}\}} e^{-\beta U(\omega)} d\omega
- \log Vol(|U-E(\beta)| <  {n\delta \over 2\beta}) + \beta E(\beta)| \leq {n\delta \over 2} 
,\end{equation}
we obtain   by combining 
(\ref{equ414}) 
and (\ref{equ415}) 
:
$$| \log Vol(|U-E(\beta)| < {n\delta \over 2 \beta} 
 ) - (\log \Lambda(\beta)+ \beta
E(\beta))|  \leq  {n\delta \over 2}  +\gamma_n,$$ from which 
(\ref{equ413})  clearly follows.

\bigskip
\textsc{Example: The ideal gas}
%\medskip
The energies $u_i= {|\theta_i|^2 \over 2m}$
of the particles of the ideal gas are i.i.d. random variables under the canonical 
laws. Therefore, due to the classical LLN concerning i.i.d.
random variables, the hypothesis (\ref{equ411})  for the TLD is  
satisfied automatically:
$$\lim\limits_{n \to \infty}P(|n^{-1}U-{3 \over 2\beta}| <  \delta|\beta)=1\ \hbox{for all}\
\delta > 0.$$
(Recall from (\ref{equ218a}) and (\ref{equ219}) the formula for the internal energy.)

In view of formula (\ref{equ2110})  for the entropy of the ideal gas, the TLD (\ref{equ412}) obtains the form 
$$\lim\limits_{\delta \to 0}\lim\limits_{n \to \infty}
|n^{-1}\log Vol (|n^{-1}U-{3 \over 2\beta}| <  \delta)+{3 \over 2} \log \beta- {3 \over 2}\log 2\pi em| < n\delta \ \hbox{eventually}.$$

\subsection{Characterization of possible equilibrium prices}

\label{sub42}

In order to be able to formulate the economic Second Law as a  TLD we need the following 
characterization of possible equilibrium prices. 
%Recall for this lemma also the definition of the conjugate parameter $\alpha(p)$, cf. . 

We say that a random variable $X \in R^l$ is {\it degenerate} if there
is a lower dimensional affine hyperplane $H \subset R^{l'}$ with $l' < l$
such that $P(X \in H)=1$. Otherwise we call $X$ {\it non-degenerate.}

The {\it support} of a random variable $X$ is defined as the minimal
topologically closed set $F$ such that $P(X \in F)=1$.
\medskip
 
\begin{theorem}\label{Theorem 1.} 
%{\bf Theorem 1.}
Let $p$  be an arbitrary price belonging the interior
$$\mathring{S}^l\ \dot=\ \{p \in S^l:\ p^j >0 \ \hbox{for all}\
j=1,...,l+1\}$$ of the price simplex $S^l$.
Suppose that the excess demand $Z(p)$ is non-degenerate.
Then the following four conditions are equivalent: 
 
\noindent (i)  $p$ is a possible equilibrium price;
\medskip

\noindent (ii) $0$ belongs to the  topological interior of the convex hull of the
support of the distribution of the total excess demand;
\medskip

\noindent (iii) there exists a conjugate parameter $\alpha(p) \in R^l$
satisfying 
(\ref{equ233});
\medskip 

\noindent (iv) price $p$ is an expected equilibrium price under
the canonical probability law:                  
$$ E(Z(p)|p) \ \dot=\  \int Z(\omega;p)P(d\omega|p)= \int Z(\omega;p)
g(\omega|p)P(d\omega)=0.$$
\end{theorem}
%\medskip
%{\it Proof:}
\begin{proof}    
Suppose that  $p$ is  a possible
equilibrium price, i.e., there exists a strictly positive probability
density function $g(\omega;p)$ such that $p$ is an expected equilibrium
price under the transformed probability law $P_g(d\omega;p) \dot=
g(\omega;p)P(d \omega)$. Since $P$ and $P_g$ are mutually absolutely
continuous (as measures, cf. \cite{bil}: p. 422), it follows that $Z(p)$ is non-degenerate under $P_g$, too. It is well known that the expectation of a non-degenerate random
variable belongs to the topological  interior
of the convex hull of the support of the distribution
of the random variable. This proves (ii).

Suppose now that (ii) holds true.
The derivative of the cumulant generating function ($\dot=$ the
 logarithm of the Laplace transform) of a random variable
defines a bijection between the domain  of the c.g.f. and the interior
of the convex hull of the support of the  random variable  (\cite{roc}: Theorem 26.5, \cite{aze}:
Proposition 9.7, \cite{ney}).
Thus it follows that $0$ belongs to the range of the
derivative of the c.g.f. $\log \Lambda(\alpha;p),$
i.e., 
 a conjugate parameter $\alpha(p) \in R^l$ exists.

That (iv) follows from (iii) was proved already in (\ref{equ324}).

The implication (iv) $\Rightarrow$ (i) is trivial. 
\end{proof}

\subsection{The Second Law of stochastic equilibrium economics as a Theorem of Large Deviations}

\label{sub43}

We are now able to formulate the Second Law as a
 Theorem of Large Deviations (TLD) concerning the random equilibrium prices.

To this end, let  $p$  be an arbitrary possible equilibrium  price belonging to the
 interior $\mathring{S}^l $ of the price simplex $S^l$.
We postulate  three  hypotheses:
\medskip

(i) The random total excess demand $Z(p)$  satisfies the  (weak) Law of
Large Numbers  under the canonical probability law $P(d\omega |p)$, i.e.,
$$\lim\limits_{n \to \infty}P(|n^{-1}Z(p)| <  \varepsilon|p)=1\ \hbox{for all}\
\varepsilon > 0. $$
Recall from Theorem \ref{Theorem 1.} %1
that $E(Z(p)|p)=0$.
\medskip

(ii)  The second derivative of the
mean total
excess demand $n^{-1}Z(q)$   is bounded on some closed  neighborhood $\overline U$ of the price $p$;
namely, there is a constant $A_2 < \infty$ such that
$$|n^{-1}Z''(q)|  < A_2\ \hbox{for }\  q \in  \overline U.$$

(iii) The derivative (matrix) $Z'(p) \in R^{l \times l}$ is invertible
and, moreover, the inverse of the derivative of the mean total excess
demand is bounded; namely, there is a constant $A_{-1} < \infty$ such that
$$  |nZ'(p)^{-1}|  <   A_{-1}.$$
%\medskip
%{\bf Theorem 2.} 
\begin{theorem}\label{Theorem 2.}
    
Under the stated hypotheses (i) - (iii), for any $\delta > 0$,
$$| {\cal I}(\exists p^*: |p^*-p| < \delta|) - \log \Lambda (p)|  <
\varepsilon (\delta)n\ \hbox{eventually},$$  where
 $\varepsilon (\delta) \to 0$ as $\delta \to 0$.
\end{theorem}
\medskip
\textsc{ Remark.} Clearly, the statement of Theorem \ref{Theorem 1.} implies that
$$\lim\limits_{\delta \to 0}\limsup\limits_{n \to \infty} |n^{-1}[{\cal
I}(\exists p^*: |p^*-p| < \delta)+\log \Lambda(p)]| =0.$$
%\medskip
%{\it Proof of Theorem 2.} 
\begin{proof} The first half of the proof is analogous to
the proof of the thermodynamic TLD, cf. Section \ref{sub41}.

First note that, due to (i), $$\gamma_n\ \dot=\ \log P(|Z(p)| <
n\varepsilon|p) \to 0 \ \hbox{as}\ n \to \infty.$$
On the other hand, in view of the definition 
(\ref{equ323}) 
of the canonical macroeconomic probability law, we can also write
\begin{equation}
\label{equ431} 
\gamma_n = \log  \int\limits_{\{|Z(p)| < n\varepsilon\}} e^{\alpha
\cdot Z(\omega;p)} P(d\omega) -\log \Lambda(p)
.\end{equation}  

Since clearly, \begin{equation} \label{equ432}|\log \int\limits_{\{|Z(p)| < n\varepsilon\}} e^{\alpha
\cdot Z(\omega;p)} P(d\omega)  - \log P(|Z(p)| < n\varepsilon) | \leq
|\alpha|n \varepsilon ,\end{equation}
we obtain   by combining  (\ref{equ431}) and (\ref{equ432}):
\begin{equation}
\label{equ433} |\log P(|n^{-1}Z(p)| < \varepsilon) - \log \Lambda (p)|  \leq
|\alpha| n \varepsilon +\gamma_n.\end{equation}

Now, it can be proved by using the standard  mean value theorem and a
special  inverse function theorem (see 
\cite{num4}: Lemma 2.5)
that the
hypotheses (ii) and (iii) imply that  the mean total excess demand
$n^{-1}Z(p)$ is in a neighborhood of $0$  if and only    there is a random equilibrium price in a neighborhood of $p$, see \cite{num4}: the
proof of Theorem 2.2. Therefore, it is
possible to deduce from (\ref{equ433})  that
$$| \log P(\exists p^*: |p^*-p| < \delta|) - \log \Lambda (p)|  <
\varepsilon (\delta)n\ \hbox{eventually},$$  where
 $\varepsilon (\delta) \to 0$ as $\delta \to 0$.
This  proves the asserted TLD.  \end{proof}
\bigskip

\textsc {Remarks:}
(i)  Hypothesis (i) is the analog of  (\ref{equ411}) 
for the thermodynamic TLD.
\medskip

(ii) In terms of probabilities, the statement of the TLD can be written as:
$$ \Lambda(p)e^{-n \varepsilon(\delta)} < P(\exists p^*: |p^*-p| <
\delta)  < \Lambda(p)e^{n\varepsilon(\delta)}\
\hbox{eventually.}$$ 
 If only the right  hand inequality holds true, viz. $$P(\exists p^*:
|p^*-p| < \delta)  <  \Lambda(p)e^{n \varepsilon(\delta)}\
\hbox{eventually}, $$ then the {\it LD upper  bound} is said to hold  true at $p$.
The LD upper bound alone implies the Law of
Large Numbers  for the random equilibrium prices (\cite{num4}: Theorem 3.5).
\medskip

(iii) For  an alternative weaker version of hypothesis (iii), see \cite{kuu}.
\medskip

(iv) According to the TLD concerning {\it partial observations}, under appropriate
regularity conditions, we have
$$\lim\limits_{n \to \infty} n^{-1}|{\cal I}(\exists p^* \in
B)+ \log \Lambda (p^*_B)|=0,$$  where  $B \subset S^l$ is convex and
$p^*_B$ denotes the entropy-minimizing price in $B$, cf.  \cite{num4}: Theorem
3.4. The principle of minimum entropy  for partial observations (cf. Remark (iii) in 2.3) is obtained from this at the ideal limit $n =\infty.$

\subsection{The Theorem of 
Large Deviations for ideal random economies}

\label{sub44}

The following corollary of Theorem \ref{Theorem 1.}  gives  sufficient conditions for a price $p$ to be a
possible equilibrium price in an  ideal random economy.

Let $\Theta$ denote the set of parameter values at which the microeconomic
p.d.f.  is strictly positive:
$$\Theta \ \dot =\ \{\theta \in R^m;\ f(\theta)>0\}.$$ Its topological closure $\overline \Theta$ equals the support of the economic characteristics $\theta_i$ of the economic agents, cf. Section \ref{sub42}.  
\medskip

%{\bf Corollary 1.}
\if 0 
ESAN TEKSTISSA KAIKKI SEURAAVAT 'TEOREEMAT' OVAT 'KOROLLAAREJA', JA SITA TAPAA OLISI HYVA NOUDATTAA TASSAKIN. EN KUITENKAAN TIEDA, KUINKA SE TEHTAISIIN. VARMAAN ALUSSA TAYTYY MAARITELLA JOKU VASTAAVA TOIMINTO, VAI KUINKA? TEIN YRITYKSEN (KATSO RIVI 42), MUTTA EN ONNISTUNUT. MYOS KAIKKI VIITTAUKSET NAIHIN KOROLLAAREIHIN TAYTYY TARKISTAA!
\fi 

\begin{corollary}\label{Corollary 1.}
Let $p$ be an arbitrary price belonging to the
 interior $\mathring{S}^l$ of the price simplex. Suppose that

\medskip
\noindent (i)  the economic characteristics $\theta_i$ are bounded
random variables of dimension  $m \geq  l$;
\medskip

\noindent (ii) the microeconomic p.d.f. $f(\theta)$ is   continuous on
its  support  $\overline \Theta;$
\medskip

\noindent (iii) the derivative
${\partial z(\theta;p) \over
\partial \theta} $  of
the structure function is  continuous
on  $ \overline \Theta \times \mathring{S}^l$; and

\medskip
\noindent (iv) there exists a parameter $\theta = \theta(p) \in \Theta$
such that $z(\theta(p);p) = 0$ and
 the  matrix  ${\partial z(\theta;p) \over \partial
\theta} \in R^{m \times l} \ \hbox{has full rank} \ (=l) \ \hbox{at}\
\theta =\theta(p).$
\medskip

Then $p$ is a possible equilibrium price.
\end{corollary}

\begin{proof} From (i) and (ii) follows that the support $\overline \Theta$ is compact, and  $\Theta$ is its (relatively) open subset.
Furthermore, because of (iii) and (iv), and by applying the  {\it implicit
function theorem} (see e.g. \cite{lan}), the structure function function $z(\theta;p)$
is  invertible in a neighborhood of $\theta(p)$, i.e., there is a constant
$\delta > 0$ such that for $|z| < \delta$ there exists
$\theta = \theta(p,z) \in \Theta$ such that
$$z(\theta(p,z);p) \equiv z.$$
Moreover, the inverse $\theta(p,z)$ is a continuous function of its variables $p$ and $z$.
Let $\varepsilon > 0$ be an arbitrary positive constant. Due to the
continuity of the structure function  $z(\theta;p)$, implied by (iii),
there is a constant $ \eta_{\varepsilon} > 0$ such that
$|z(\theta;p)- z| < \varepsilon$ whenever $|\theta - \theta(p,z)| <
\eta_{\varepsilon}$. From (ii) follows
that $f(\theta)$ is strictly positive in a neighborhood of $\theta(p,z)
\in \Theta$, whence  for any $|z|<\delta$,
$$P(|z(\theta_i;p) -z| < \varepsilon )
\geq P(|\theta_i -\theta(p;z)| < \eta_{\varepsilon} ) 
= \int\limits_{\{|\theta -\theta(p;z)| < \eta_{\varepsilon}\}} f(\theta)
d\theta > 0,$$
i.e., $0$ belongs to the topological interior of the support of the
distribution of the individual excess demand $\zeta_i(p)=z(\theta_i;p)$.
Recalling the argument used in the proof of Theorem \ref{Theorem 1.}, it  follows that $0$ belongs to the range of the
derivative ${\partial \over \partial \alpha}\log \lambda (\alpha;p)$. In
view of the same Theorem and the relation 
(\ref{equ246}), $p$ is  a possible equilibrium price.  \end{proof}

\medskip 
In the following corollary we formulate a set of sufficient conditions for the TLD  to hold true
for an ideal random economy:
\medskip

\begin{corollary}\label{Corollary 2.}  Let $p \in \mathring{S}^l$ be arbitrary. Suppose that the
conditions (i) - (iv) for Corollary \ref{Corollary 1.} hold true (so that $p$ is a possible
equilibrium price). 
Moreover, suppose that
\medskip

\noindent (v) the derivative
 ${\partial^2 z(\theta;p) \over \partial p^2}$  is  continuous
on  $ \overline \Theta \times \overline U$ for some closed neighborhood $\overline U \subset \mathring{S}^l$ of $p$; and

\medskip
\noindent (vi)  $z(\theta;p)$ is a
bounded perturbation of a deterministic function $z(p)$  (i.e., $z(p)$
does not depend on the parameter $\theta$) in the following sense:
$$\rho\ \dot=\ \max\limits_{\theta \in \overline \Theta}|{\partial z
\over \partial p} (\theta;p) z'(p)^{-1}-I| < 1.$$

\medskip
Then the TLD holds true at price $p$.
\end{corollary}

\begin{proof} We  show that the  hypotheses (i) - (iii) for the general
TLD (Theorem \ref{Theorem 2.}) hold true.

Recall that under the  canonical macroeconomic probability law the
economic characteristics are i.i.d. and obey the canonical
p.d.f. $f(\theta|p)$. It follows that, as their deterministic transforms, also the microeconomic random variables $\zeta_i(p) = z(\theta_i;p)$  are i.i.d. Thus hypothesis (i) for the general TLD follows directly from the LLN for i.i.d. random variables.

Condition  (v) implies that ${\partial^2 z(\theta;q)
\over \partial q^2}$ is bounded by some constant $A_2$ on the compact set $\overline \Theta
\times \overline U$.
Therefore
$$|Z''(q)| =|\sum\limits_{i=1}^n {\partial^2 z(\theta_i;q) \over
\partial q^2}| \leq
A_2n\ \hbox{on}\   \overline U.$$ Thus hypothesis (ii) for the
TLD holds true.

In view of (vi),  we have $$|n^{-1}Z'(p)z'(p)^{-1} -I| = |n^{-1} \sum\limits_{i=1}^n({\partial z \over \partial p}(\theta_i;p)z'(p)^{-1}-I| \leq \rho,$$ whence $$|Z'(p)^{-1}| < {|z'(p)^{-1}| \over n(1-\rho)},$$ i.e., hypothesis (iii) for the TLD is satisfied for $A_{-1} =  {|z'(p)^{-1}| \over 1-\rho}$.  
\end{proof}
\medskip

\textsc{Remarks:}
%\medskip
(i) Condition (vi) can be weakened  considerably  (\cite{kuu}). Namely, it suffices to
assume that the the derivative $\mu'(p)$ of the individual expected excess
demand  is non-singular.  Clearly, in the  Cobb-Douglas case
$\hbox{Det}\ \mu'(p)=0$ is a polynomial equation, and therefore
$\mu'(p)$  is non-singular except for a set of prices $p$ having
Lebesgue measure zero.
\medskip

(ii)  Suppose that $l=1$ so that $p$ and $z(\theta;p)$ are scalars. It
is natural to assume that  $z'(p) < 0$. In this case  condition (vi) becomes $$(2-\delta) z'(p) \leq {\partial z \over \partial p}(\theta;p)  \leq \delta z'(p)\ \hbox{for some}\ \delta > 0.\leqno\hbox{(vi')}$$

\bigskip
%\textsc{ Example:}
%\medskip
The following further corollary of Corollary \ref{Corollary 1.} gives sufficient conditions for a
price $p \in \mathring{S}^l$  to be a  possible equilibrium price in an ideal CD economy. Let
$$\Theta \ \dot=\ \{\theta=(a,e) \in S^l
\times R^{l+1}:\ f(a,e) >0\}.$$

%{\bf Corollary 3. }
\begin{corollary} \label{Corollary 3.}
Suppose that
\medskip

\noindent (i) the initial endowments $e_i$ are
 bounded;
\medskip
    
\noindent (ii) the microeconomic p.d.f. $f(a,e)$ is continuous  on the  support
$\overline \Theta $; and
\medskip

\noindent (iii)  the microeconomic p.d.f. $f(a,e)$ is strictly positive at $a=p,\
e \equiv 1$, (i.e.,    $\theta(p) \ \dot=\ (p,1) \in \Theta$).

Then  $p$ is a possible equilibrium price.
\end{corollary}  

\begin{proof} We  verify the conditions (i)-(iv) for Corollary \ref{Corollary 1.}:

For an ideal random CD economy, the dimension $m$ of the economic
characteristics $\theta  = (a,e)$ is  $2l+1$, which is  $> l$ as required for Corollary \ref{Corollary 1.}. Since the share parameters  belong to the (bounded) simplex $S^l$, it follows that the economic characteristics $\theta_i = (a_i,e_i)$ are bounded as required.

Clearly, the structure function  (\ref{equ248}) of a CD agent 
is infinitely many times differentiable
w.r.t. its variables $a \in S^l$, $e \in R^{l+1}$ and $p \in \mathring{S}^l$. 
Thus condition (iii) for Corollary \ref{Corollary 1.} is satisfied automatically.

A direct calculation shows that $$z(\theta(p);p) = z(p,1;p) =0.$$ Moreover,  as is easy to see,  already the derivative $${\partial z(a,e;p) \over \partial a} \in R^{(l+1) \times l})$$  has the (full) rank ($=l$) at $ \theta(p) = (p,1)$. Thus also  condition (iv) for Corollary \ref{Corollary 1.} is satisfied.
\end{proof}
%\medskip
 
In the following corollary we formulate a set of sufficient conditions for the TLD  to hold true for an ideal random CD.  
 For simplicity we assume  that
$l=1$.

%{\bf Corollary 4.}
\begin{corollary} \label{Corollary 4.}Suppose that $l=1$ and that the  conditions (i)-(iii) for Corollary \ref{Corollary 3.} hold true. In
addition, suppose that
\medskip 

\noindent  (iv) the initial endowment of the commodity $2$ is
bounded  away from zero and bounded from above, and that the share parameter
of commodity $1$ is
bounded away from zero, i.e.,
$ \underline e^2 \leq  e^2  \leq \overline  e^2 ,\
\hbox{and}\ a^1 \leq \underline a^1  $ for some constants
$ \underline e^2 > 0, \   \overline  e^2 < \infty,\  \underline a^1
> 0.$
 
Then the TLD holds true.
\end{corollary}

\begin{proof} We  check that the conditions (i)-(v) and (vi') for Corollary \ref{Corollary 2.} are satisfied.

The conditions (i)-(iv) for Corollary \ref{Corollary 2.}  were verified already in the proof of Corollary \ref{Corollary 3.}. Condition (v) follows  from the smoothness of the structure function of a CD agent in the interior of the price simplex, cf. the proof of Corollary \ref{Corollary 3.}.

In order to see that condition (vi') is satisfied, note first that
$$ {\partial z \over \partial p}(a,e;p)  =   -p^{-2}a^1 e^2.$$
Let $z(p)$ denote the individual excess demand by a deterministic CD agent having parameters $a=({1 \over 2},{1 \over 2}),\ e=(1,2\overline e^2)$ so that $$z(p) =   {1-p \over p} \overline e^2 -{1 \over 2}$$ and $$z'(p) = -p^{-2} \overline e^2.$$
A straightforward calculation now shows that condition (vi') is satisfied for $$\delta = { \underline a^1 \underline e^2 \over \overline e^2}.$$
\end{proof}

\section{Conditional Laws of Large Numbers}

\label{sec5}

\subsection{Thermodynamic Gibbs Conditioning Principle as a conditional Large Numbers}

\label{sub51}

Let $\beta > 0 $ be a fixed inverse temperature.

Recall that the measurement of the inverse temperature
at $\beta$ means that the particle configuration $\omega$ belongs
to the  energy shell
$\{|U-E(\beta)| < \Delta\}$,
where $E(\beta)$ denotes the {\it internal
energy} at $\beta,$ and $\Delta$ denotes the thickness of the
energy shell. 
In the exact formulation of the conditional LLN, $\Delta = n \delta$, where $n$  is the total number of 
particles and %$\delta $ is "small". (In the conditional LLN we let $n \to \infty$ and then 
$\delta \to 0$.

A variable      $X=\{X(\omega);\omega \in \Omega\}$, which depends  on
the particle configuration $\omega \in \Omega$, is called a
{\it thermodynamic variable}.  
Such a variable can be regarded as
a random variable under the canonical probability 
laws
$P(d\omega|\beta),\ \beta >0.$ 

As a theorem, the thermodynamic
Gibbs Conditioning Principle is a      {\it
conditional Law of Large Numbers}
concerning thermodynamic variables.

\medskip
We call a thermodynamic variable  $X$
 {\it regular} (at $\beta$) if the
 (weak) LLN  holds true for $X$
under the  canonical probability law 
$P(d\omega|\beta)$
with a {\it   geometric rate,} i.e., for all $\varepsilon >0$ there
exists a constant $\eta=\eta(\varepsilon;\beta)>0$ such that
\begin{equation}
\label{equ511}
P(n^{-1}|X-E(X|\beta)| < \varepsilon| \beta)) >1-e^{-n\eta}
\ \hbox{eventually.}
\end{equation}
%\medskip

According to the conditional LLN, conditionally on the measurement of the inverse temperature at $\beta$, regular thermodynamic variables are {\it centered} at their canonical means; namely, the  proportion of the total volume of those particle configurations $\omega$  in the energy shell $\{|U-E(\beta)| < \Delta\}$,  where the thermodynamic variable $X(\omega)$ is near to its canonical expectation $E(X|\beta)$, is near to $1$.

In exact terms, the conditional LLN is as follows:

\medskip
Suppose that the statement (\ref{equ413}) of the thermodynamic TLD holds
true at $\beta $, i.e., 
\begin{equation}
\label{equ512}
\lim\limits_{\delta \to 0}\limsup\limits_{n \to \infty}
|n^{-1}[\log \hbox{Vol} (|U-E(\beta)| < n\delta)
-(\log \Lambda(\beta)+ \beta E(\beta))]| =0.\end{equation}

Let $X$ be an arbitrary regular thermodynamic variable and let
$\varepsilon  > 0$ be an arbitrary constant.
Then for all sufficiently small $\delta > 0$,
\begin{equation}
\label{equ513}
\lim\limits_{n \to \infty}
{\hbox{Vol} (|X-E(X|\beta)| <n\varepsilon,|U-E(\beta)| < n\delta) \over
\hbox{Vol} (|U-E(\beta)| < n\delta)} =1.\end{equation}
\medskip
 
In order to prove this, 
note first that by inverting the defining formula 
(\ref{equ311}) 
of the canonical probability law, we can express the 
volume of any (Borel) subset $A \subset \Omega$ 
with the aid of the canonical probability law:
$$\hbox{Vol}(A) = \int_A d\omega
= \Lambda(\beta) \int_A e^{\beta U(\omega)}
 P(d\omega|\beta).$$

Let  $\varepsilon >0$ be an arbitrary constant, and let $\eta =
\eta(\varepsilon;\beta)$ be such that (\ref{equ511}) 
holds true. Moreover, let $0 < \delta < {\eta \over 2\beta
}$ be arbitrary. We can now write
\begin{eqnarray}
\hbox{Vol} (|X-E(X|\beta)|\geq n\varepsilon,&|& U-E(\beta)|< n\delta
) \nonumber \\
&=& \Lambda(\beta) \int\limits_{\{|X-E(X|\beta)|\geq n\varepsilon,\
|U-E(\beta)|< n\delta 
\}}e^{\beta
U(\omega)}P(d\omega|\beta ) \nonumber \\
&\leq & \Lambda(\beta) e^{\beta E(\beta) +\beta \delta n}
P(|X-E(X|\beta)|\geq
n \varepsilon|\beta) 
\nonumber \\
&<& \Lambda(\beta) e^{\beta E(\beta) +\beta \delta n}e^{-n\eta } \
\hbox{eventually, by (\ref{equ511}),}
\nonumber \\
&<& \Lambda(\beta) e^{\beta E(\beta) } e^{-n{\eta \over 2}} \ \hbox{
because}\ \delta < {\eta \over 2 \beta }. \nonumber 
\end{eqnarray}

Now, according to the hypothesis 
(\ref{equ512})
:
$$ \hbox{Vol}
(|U-E(\beta)| < n\delta) > \Lambda(\beta)e^{\beta E(\beta)}e^{-n{\eta
\over 3}}  \ \hbox{eventually},$$
if $\delta$ is sufficiently small. Therefore, for sufficiently small
$\delta$ we  have eventually
$${\hbox{Vol} (|X-E(X|\beta)| \geq n\varepsilon,|U-E(\beta)| 
< n\delta) \over
\hbox{Vol} (|U-E(\beta)| < n\delta)}
< { \Lambda(\beta)e^{\beta E(\beta)}e^{-n{\eta \over 2}}  \over
\Lambda(\beta)e^{\beta E(\beta)}e^{-n{\eta \over 3}}} 
= e^{-n{\eta \over 6}}.$$
This proves that (\ref{equ513}) holds true (with the convergence being at a geometric rate). 
\bigskip

\textsc{Example: The ideal gas.}
%\medskip
Recall that the TLD is true for the ideal gas (Section \ref{sub41}).
Therefore the conditional LLN holds true
automatically for any regular thermodynamic variable.

There is a natural class of regular thermodynamical variables 
for the ideal gas:

To this end, let $A  \subset \R^3$ be an arbitrary
(Borel-measurable) subset of $\R^3,$ and let $\chi_A(\theta) \dot= 1$ or
$0$, according as $\theta \in A$ or $\theta \in A^c$, denote the 
indicator function of $A$.

Recall that under the  canonical probability law the
momenta $\theta_i $ of the particles are   i.i.d.  random
variables.
Therefore their (deterministic) transforms  $\chi_A(\theta_i)$ form also an i.i.d.
sequence.

Let us define  the thermodynamic variable $N_A$  as the number of
particles $i$ having momentum $\theta_i \in A:$
$$N_A(\omega)\ \dot=\ \sum\limits^n_{i=1}\chi_A(\theta_i).$$ Now it
follows that $N_A$ is automatically regular.
This is due to a general result, according to
which  for bounded i.i.d. random variables the convergence in the LLN is
always geometric (see e.g. \cite{demzei}: Section 2.3). %ex2.3...
Clearly $$E(N_A|\beta) = n P(\theta_i \in A|\beta) 
= n \int\limits_A f(\theta|\beta),$$ where $f(\theta|\beta)$ 
denotes the Gaussian canonical p.d.f. given by (\ref{equ313}).

Let $\hat n_A \dot=
n^{-1}N_A $ denote the proportion  of
particles $i$ having momentum $\theta_i$.
It follows that the conditional LLN  for the variable $N_A$ obtains the form
\begin{equation}
\label{equ514}
\lim\limits_{n \to \infty}
{\hbox{Vol} (|\hat n_A-\int\limits_A f(\theta|\beta)| <
\varepsilon,|U-ne(\beta)| < n\delta) \over
\hbox{Vol} (|U-ne(\beta)| < n\delta)} =1,
\end{equation}
where $e(\beta) = {3 \over 2 \beta}$ denotes the internal energy of a single particle (see (\ref{equ219})) and 
$\varepsilon$ and $\delta$ are small in the same sense as in (\ref{equ513}).
 Thus, at the ideal limit $\varepsilon = 0$,
in accordance with the Gibbs Conditioning Principle,
at a measured inverse temperature, for "most" particle
configurations, the proportion of particles with momentum in
$A$ equals the canonical probability of $A$.

In the standard terminology of probability and statistics, the proportions $\hat n_A,\ A \subset R^3,$
form the {\it empirical probability distribution} of the momenta $\theta_i.$ The result  (\ref{equ514}) of the conditional LLN means that, at the thermodynamical limit $n \to \infty$, the empirical probability distribution $\hat n$ converges to the canonical probability distribution associated with the measured inverse temperature.

\subsection{The economic Gibbs Conditioning Principle as a conditional Law of Large Numbers}

\label{sub52}

As a theorem, the  economic Gibbs conditioning principle is a {\it conditional Law of Large Numbers } concerning {\it macroeconomic random variables.}

\medskip
Let $p \in \mathring{S}^l$ be a fixed possible equilibrium price (see Definition \ref{def2} in Section \ref{sub32}). We postulate three hypotheses (i) - (iii):
\medskip

\noindent (i) The statement of the Theorem of Large Deviations holds
true at $p$, i.e.,
$$\lim\limits_{\delta \to 0}\limsup\limits_{n \to \infty} |n^{-1}[{\cal
I}(\exists p^*: |p^*-p| < \delta)+\log \Lambda(p)]| =0,$$ cf. Theorem \ref{Theorem 2.}

\medskip

Let $X(p)=\{X(\omega;p);\omega \in \Omega\}$ be a macroeconomic random variable, viz., a random variable which depends on price $p$. We   call it {\it regular (at  price $p$)} if the following condition is satisfied (cf. (\ref{equ511})):
\medskip

\noindent (ii) The  Law of Large Numbers  holds true for $X(p)$
 under the  canonical probability law $P(d\omega|p)$
with  {\it   geometric rate:} for all $\varepsilon >0$ there
exists a constant $\eta=\eta(\varepsilon;p)>0$ such that
$$P(n^{-1}|X(p)-E(X(p)|p)| < \varepsilon| p)) >1-e^{-n\eta}
\ \hbox{eventually.}$$

\medskip
Moreover, we assume the following:
\medskip

\noindent (iii) The derivatives of the mean total excess demand and of the mean of the macroeconomic variable are bounded on some closed neighborhood $\overline U$ of price $p$;
namely, there is a constant $A < \infty$ such that
$$|n^{-1}Z'(q)|  \leq A \ \hbox{on}\    \overline U 
\ \hbox{and} \ |n^{-1}X'(q)|  \leq A \ \hbox{on}\    \overline U.$$
\medskip

Let $X^*=X(p^*)$ denote the {\it value} of the macroeconomic random variable $X(p)$ {\it at the equilibrium.}
\medskip

According to the  conditional LLN, conditionally on the observation of a
random equilibrium price at price $p$, regular
 macroeconomic random variables are centered at values at which they would
 be centered if they obeyed the canonical probability law associated with the observed equilibrium:

\medskip
%{\bf Theorem 3.} 
\begin{theorem}\label{Theorem 3.}
Under the hypotheses (i) - (iii), for any fixed
$\varepsilon > 0$ and all sufficiently small $\delta > 0$:
$$\lim\limits_{n \to \infty}               P(
n^{-1}|X^*-E(X^*|p)|< \varepsilon|\exists p^*: |p^*-p| < \delta) =1.$$
\end{theorem}

\begin{proof} By inverting the defining formula (\ref{equ323})  of the canonical probability law we can express the a priori macroeconomic
probability law with the aid of the canonical probability law:
$$P(d\omega)= \Lambda(p) e^{-\alpha(p) \cdot Z(\omega;p)}P(d\omega|p).$$

Let  $\varepsilon$ and $\gamma >0$ be arbitrary constants. ($\gamma $ will be fixed soon.)  We can now write
\begin{eqnarray}
P(n^{-1}|X(p)&-&E(X(p)|p)| \geq \varepsilon, \ n^{-1}|Z(p)| < \gamma) \nonumber \\ &=&
\Lambda(p) \int\limits_{\{n^{-1}|X(p)-E(X(p)|p)|\geq
\varepsilon,\ n^{-1}|Z(p)|< \gamma\}}e^{-\alpha(p) \cdot
Z(\omega;p)}P(d\omega|p) 
\nonumber \\
&\leq&
 \Lambda(p) e^{n|\alpha(p)| \gamma} P(n^{-1}|X(p)-E(X(p)|p)|\geq
\varepsilon|p) 
\nonumber \\
&<&
\Lambda(p) e^{n|\alpha(p)| \gamma}e^{-n\eta } \
\hbox{eventually, by (iii),} 
\nonumber \\
&=& 
\Lambda(p) e^{-n{\eta \over 2}} \ \hbox{if we choose}\ \gamma =
{\eta \over 2|\alpha(p)|}.
\end{eqnarray}
Now, due to hypothesis (iii) and the standard {\it mean value
theorem}, we can conclude that, if $\delta >0$ is sufficiently small, then the
event $\exists p^*: |p^*-p|< \delta$ implies the event $$ |n^{-1}Z(p)|<
\gamma.$$ 
Therefore, for sufficiently small $\delta$, we have also
$$ P( n^{-1}|X(p)-E(X(p)|p) |\geq
\varepsilon,\ \exists p^*: |p^*-p|< \delta)
< \Lambda(p) e^{-n{\eta \over 2}}  \ \hbox{eventually}.$$
Now, according to the hypothesis (i):
$$ P(\exists p^*: |p^*-p|< \delta) > \Lambda(p)e^{-n{\eta \over 3}}  \ \hbox{eventually}.$$
Therefore, for sufficiently small $\delta$, we have eventually
\begin{eqnarray} P(n^{-1}|X(p)
&-& 
E(X(p)|p)|\geq \varepsilon| \
\exists p^*: |p^*-p|< \delta) \nonumber \\ &=& {P(n^{-1}|X(p)-E(X(p)|p)|\geq
\varepsilon,\ \exists p^*: |p^*-p|< \delta)
\over P(\exists p^*: |p^*-p|< \delta)} 
\nonumber \\ &<& 
{ \Lambda(p)e^{-n{\eta \over 2}}  \over  \Lambda(p)e^{-n{\eta \over 3}}}
\nonumber \\ &=& 
 e^{-n{\eta \over 6}}.
\end{eqnarray}

In order to complete the proof, note first that, due to (iii) and the mean value theorem,  if $ |p^*-p|< \delta$,
then    $$|n^{-1}X(p^*)-n^{-1}X(p)| \leq A \delta.$$
Therefore, if $\delta < {\varepsilon \over A}$ is sufficiently small,
then
 
 $$ P(n^{-1}|X^*-E(X^*|p)| \geq 3 \varepsilon| \ \exists p^*: |p^*-p|<\delta)$$
$$< P(n^{-1}|X(p)-E(X(p)|p)| \geq  \varepsilon| \ \exists p^*: |p^*-p|<
\delta).$$  

As was proved above, this expression $\to 0$ as $n \to \infty$.  \end{proof}
\bigskip

\textsc{ Remarks:}
%\medskip
(i)  For a sufficient condition for hypothesis (ii), cf. Lemma 4.2 in \cite{num4}.

\medskip
(ii) Note that if $X(\omega;p) \equiv X(\omega)$ is a macroeconomic random variable, which does not depend on price, then hypothesis (iii) is trivially true for it. In this case, the conditional LLN obtains the form $$\lim\limits_{n \to \infty} P(n^{-1}|X-E(X|p)|< \varepsilon|\exists p^*: |p^*-p| < \delta) =1.$$

\subsection{The conditional Law of Large Numbers for ideal random economies}

\label{sub53}

Consider an ideal random economy as described in Section \ref{sub24}. A  random variable $\xi_i(p)= x(\theta_i;p)$, which depends via a deterministic {\it structure function} $x(\theta;p)$ on the economic characteristics $\theta_i$ and on  price $p$, is called a {\it microeconomic random variable} (associated with the agent $i$). Examples are, e.g.,  the agent's individual demand, supply, production or  type (cf. the example below). Recall that statistical independence is preserved under deterministic  transformations. Therefore,  the microeconomic r.v.'s $\xi_i(p)$ are i.i.d., too.

Let us define the macroeconomic random variable $X(p)$ as the sum
$$X(p) \dot= \sum\limits_{i=1}^n\xi_i(p),$$ 
and let $\hat \xi(p) \dot= n^{-1}X(p)$ denote its mean.

Let $X^*  \dot= X(p^*)$ and $\hat \xi^* \dot= \hat \xi(p^*)$ denote the value and the mean, respectively, 
 of the macroeconomic random variable $X(p)$ at the equilibrium.

Clearly $$E(X^*|p)= nE(\xi_i(p^*)|p) = n \int x(\theta;p^*)f(\theta|p) d \theta.$$ 
% \medskip

The conditional LLN of Theorem \ref{Theorem 3.} obtains now the following form:

%\medskip
\begin{corollary}\label{ Corollary 5.} Suppose that  conditions (i) and (ii) for Corollary \ref{Corollary 1.} in Section \ref{sub44}
hold true. In addition, suppose that
\medskip

\noindent (iii) the TLD holds true at price $p$ (cf. Corollary \ref{Corollary 2.} in Section \ref{sub44});
and
\medskip

\noindent (iv) the derivatives ${\partial z(\theta;q) \over \partial q}
$ and  ${\partial x(\theta;q) \over \partial q}$ are continuous
on  $\overline \Theta
\times \overline U$ for some closed neighborhood $\overline U \subset \mathring{S}^l$ of  $p$.
\medskip

Then  
 $$\lim\limits_{n \to \infty}
P (|\hat \xi^* -\int\limits x(\theta;p^*)f(\theta|p)d\theta| <\varepsilon| \exists p^*:|p^*-p|< \delta) =1.$$ 
\end{corollary}

\begin{proof}  We  show that the hypotheses (i) - (iii) for Theorem \ref{Theorem 3.} hold true. 
Recall that, under the  conditions (i) and (ii) for Theorem \ref{Corollary 1.} in Section \ref{sub44}, the economic characteristics are
bounded and their support $\overline \Theta$ is compact. In view of (iv), the structure function $x(\theta;p)$ of the microeconomic r.v.'s $\xi_i(p)$ is continuous on $\overline \Theta$, rendering $\xi_i(p),\ i=1,2,...$ bounded.
Also recall that, under the  canonical macroeconomic probability law, the
economic characteristics are  i.i.d. and obey the canonical
p.d.f. $f(\theta|p)$.
The hypothesis (ii) for Theorem \ref{Theorem 3.} now follows from a general result, according to
which the convergence in the  LLN for bounded i.i.d. random variables is always geometric.

Due to condition  (iv), 
${\partial z(\theta;q) \over \partial q}$ and 
${\partial x(\theta;q) \over \partial q}$ are
bounded on the compact set $\overline \Theta
\times \overline U$.
Therefore, also 
$$|n^{-1}Z'(q)| =|n^{-1}\sum\limits_{i=1}^n {\partial z(\theta_i;q) \over \partial q}|$$ and 
$$|n^{-1}X'(q)| =|n^{-1}\sum\limits_{i=1}^n {\partial x(\theta_i;q) \over \partial q}|$$ 
are  bounded by  the same constant on $\overline U.$ Thus the hypothesis (iii) for Theorem \ref{Theorem 3.} is satisfied.  \end{proof}

%\bigskip

\textsc{Examples:}
(i) Consider the  survival model of Example (ii) in Section \ref{sub33}, and define the microeconomic random variable 
$\xi_i(p) \dot= \chi_{\{p \cdot e_i < w(p)\}}$ as the indicator of non-survival.

Let  $N(p)$ be the number of
non-surviving agents $i$ at price 
$p$: $$N(p)\ \dot=\ \sum\limits^n_{i=1}\chi_{\{p \cdot e_i < w(p)\}},$$ and let $\hat n(p) \dot=
n^{-1}N(p) $ denote the  proportion of non-surviving agents at price $p$.

The proportion $\hat n(p) \dot= n^{-1}N(p) $ can be interpreted as the {\it empirical  probability} of non-survival at price $p.$ As a sum of i.i.d. random variables the macroeconomic variable $N(p)$ is automatically regular.

Let $N^* \dot= N(p^*) $ and  $\hat n^* \dot= \hat n(p^*) $  denote the {\it equilibrium values} of $N(p)$ and $\hat n(p)$, respectively. Clearly $$E(N^*|p) = n P(p^* \cdot e_i < w(p^*)|p) = n \int\limits_{\{p^* \cdot e_i < w(p^*)\}} f(\theta|p)d\theta.$$

It follows that the conditional LLN  for the variable $N(p)$ obtains the form
$$\lim\limits_{n \to \infty} P (|\hat n^*-\int\limits_{\{p^* \cdot e_i < w(p^*)\}} f(\theta|p)d\theta|  < \varepsilon| \exists p^*:|p^*-p|< \delta) =1.$$

This means that, conditionally on the observation of an equilibrium, the empirical probability of non-survival converges to its canonical probability.

Thus, in accordance with Gibbs Conditioning Principle, at the ideal limit $\varepsilon = 0$ and at an observed equilibrium price, for "most" macroeconomic configurations $\omega$ the proportion of non-surviving agents equals the canonical probability of non-survival.
\medskip

(ii) The conditional LLN implies also the convergence of the whole {\it empirical probability distribution} of the economic characteristics (cf. the ideal gas as the analogy, Section \ref{sub51}).
%\medskip

To this end, let $A  \subset \Theta$ be an arbitrary (Borel) subset of the support $\Theta \subset R^m$ of the microeconomic p.d.f. $f(\theta)$, 
and let $\chi_A(\theta)$   denote the indicator function of $A$. Let   $N_A$ be the  number of agents $i$ with characteristics  $\theta_i \in A:$
$$N_A(\omega)\ \dot=\ \sum\limits^n_{i=1}\chi_A(\theta_i).$$

As a sum of i.i.d. bounded random variables,   $N_A$ is automatically regular. Clearly $$E(N_A|p) = n P(\theta_i \in A|p) = n \int\limits_A f(\theta|p)d\theta.$$ 

Let $\hat n_A \dot= n^{-1}N_A $ denote the proportion  of
agents $i$ with characteristics $\theta_i \in A$. The proportions 
$\hat n_A, \ A \subset \Theta,$ form the {\it empirical probability distribution}  of the economic characteristics $\theta_i$. 

According to the conditional LLN, the empirical probability distribution converges to the canonical probability distribution associated with the observed equilibrium:
$$\lim\limits_{n \to \infty} P (|\hat n_A-\int\limits_A f(\theta|p)d\theta| < \varepsilon| \exists p^*:|p^*-p|< \delta) =1,$$
cf. (\ref{equ514}). Thus, in accordance with Gibbs conditioning principle, at the ideal limit $\varepsilon = 0$ and  at an observed equilibrium price, for "most" macroeconomic 
configurations $\omega$ the proportion of agents having their economic characteristics  in $A$ equals the canonical probability
 of $A$.
% \medskip
%{\color{blue} Edella (i) nayttaisi olevan (ii):n yksinkertainen erikoistapaus. Teksti toistuu niissa kaytannollisesti katsoen identtisena. Voitaisiinko (i) poistaa ja vain (ii) sailytettaisiin?}

% \medskip
 
\bigskip

%Appendix A: General background on  LD theory and statistical mechanics \textcolor{blue}{Minusta voisi olla hyva ajatus muokata tassa alla olevista Appendix A:sta ja Appendix B:sta jonkinlainen "Discussion", joka sitten sijoitettaisiin artikkelin loppuun (ennen kirjallisuusviiteluetteloa). Nythan artikkeli loppuu ehka vahan toksahtaen, ilman tavanomaista pohdintaa siita, miten hyvin alussa asetettu tavoite on saavutettu, etc. } 

\section{A short recount}\label{appn}
We have above argued for the relevance of the theory of  large
deviations in the study of the equilibria  of random economies
comprising a large number of participating agents.

In probability theory, by a  large deviation is  meant
the occurrence of a value for a random variable that falls  
 outside the region of validity of the    Central Limit Theorem.
 Large deviations are bound to occur in a large random system if the
 a priori model is  imperfect. %(as is the case in economics, cf. Remark \ref{sub23}  (ii) and Section \ref{sub25}).

The  standard  LD theory is concerned with  the probabilities of   large
deviations of  "extensive"  random variables, 
which result from an accumulation of a large number of "micro" random
variables. The standard example is  the sum of  i.i.d.  random variables. Extensivity can also be  temporal, in which case time is regarded as the size parameter.

Theorems of LD type express  the probabilities of large
deviations in an exponential form, where the exponent  is proportional to
the size parameter of the system. The complement of the coefficient of
proportionality is referred to as the  rate function. 
Due to the  thermodynamic analogy, the rate function  is also referred
to as the entropy function. The estimate yielded by the  LD theory is valid also outside the region of validity of the CLT.

The  Gibbs Conditioning Principle is  concerned with the
a posteriori  inference from the observation  of a large deviation, i.e.,
how to take into account such an observation  in a mathematically
legitimate way. According to the GCP,  the  a posteriori probability law
belongs to the  exponential family generated by the random variable
under consideration.
%, see \cite{demzei}: p.89.
As a mathematical theorem, GCP is a  conditional law of large numbers.

%{\color{blue}Esan kommentti "tahan asti" alkuperasitekstiin. Eljan kommentti: Taman (nyt Appendix A:na) olevan tekstin editointia ei ole tehty juurikaan, ja luullakseni sita voisi, lahinna toiston valttamiseksi, lyhentaa jonkin verran nykyisesta. Myos annetut aikamaareet ovat muuttuneet vuodesta 2017. }

Since the seminal  classical work by H. Cram\'er \cite{cra}), LD theory
has become a major subject in probability theory, and subsequently also
an important tool in stochastic modelling, for example, in  statistics,
information theory, engineering problems (\cite{buc}, \cite{demzei}),
modelling of large  communication networks \cite{shwwei}, risk theory
\cite{mar2}, dynamical macroeconomic phenomena \cite{aok},  calibration
of asset prices \cite{ave}, and in analyzing large portfolio losses
\cite{demdeuduf}.

%Dating back at least half a century, 
There is a long history in the search  of an analogy between economics and
thermodynamics, see, e.g., Samuelson \cite{sam}. Inspired by the
ideas of Jaynes (\cite{jaynes1957information}, \cite{jaynes1957informationII}) on the connections between statistical mechanics and
information theory,  Foley (e.g., \cite{smifol})  has
elaborated on this analogy further.%in the past 15 years a theory of this analogy
%({\color{blue}\cite{fol1}, \cite{fol2}}, \cite{smifol}). {\color{blue} Tata pitaa varmaan tarkentaa ja ehka taydentaa, mahdollisesti Duncan Foleyn Google Scholar -sivun viitteista. Artikkeliin \cite{smifol} on viitattu nyt 180 kertaa; pitaisiko niista poimia jotakin tahan? Enta tuon artikkelin toinen kirjoittaja Eric Smith?}

%We argue here that there  is a fundamental analogy in the formalisms of economics and thermodynamics, based on the theory of large deviations and  due to the common ground of stochastic equilibrium economics and statistical mechanics in the theory of large deviations.

The theory of large deviations   can be regarded as the mathematical
abstraction of the inherent mathematical structure of  statistical
mechanics in that large stochastic systems are understood as
"thermodynamic"  systems. Due to this relationship, the proposed formalism
for stochastic economic equilibrium theory is  conceptually and
structurally similar to the formalism of  statistical mechanics. %\textcolor{blue}{Esa: Tahan, tai johnkin muuhun kohtaan, voisi varmaan lisata viitteen Mikko Kuuselan vaitoskirjaan. Olisiko sopiva teksti esim. "
The relationship has been pointed out earlier in \cite{kuu}, a PhD thesis supervised by Nummelin.

Classical thermodynamic systems are physical systems comprising a large number of particles (e.g., \cite{mar1} or
\cite{pliber}).
The   goal   is to describe the macroscopic behaviour of the system in terms of  a few   thermodynamic variables. Standard thermodynamic
variables are, e.g.,  volume, pressure, internal energy,
temperature and entropy. They are said to be  extensive if they are  proportional to the
volume,   and intensive  if they are independent from it. Examples
of extensive variables are  volume, internal energy  and entropy,  whereas pressure and temperature  are intensive.

The relationships between  thermodynamic variables are described by 
thermodynamic equations of state. Thus, e.g.,  the integral form of
the Second Law of thermodynamics relates thermodynamic entropy
to temperature, internal energy and partition function, see formula
(\ref{equ216}).

According to the   paradigm, the  thermodynamic laws are assumed to hold true  universally, i.e., to govern the behaviour of  any
thermodynamic system, and even if the behaviour of the system is not fully  understood mathematically.  

Like a thermodynamic system, a large economic system comprising a large
number of economic agents has many "degrees of freedom", only a few of
which are observable. As in statistical mechanics, one can distinguish between
extensive (proportional to the size of the economy) and intensive
(independent of the size) variables. Examples of extensive variables
are,  e.g., the total demand and supply on some commodities, whereas 
prices are typical intensive variables.

Despite their intensive character, however, as  zeros of the (extensive) random total excess demand function,  random equilibrium prices still  obey the  principles of large deviation theory, see  
\cite{num1}, \cite{num2}, \cite{num3}, \cite{num4}. 
Therefore,  stochastic equilibrium theory is a  potential
area for applying the LD methods.

%We argue that, in analogy with 
As in statistical
mechanics, the  "thermodynamic formalism" of stochastic equilibrium
theory  reflects certain  underlying universal principles. Their validity is not restricted to only those  random economies for which the exact mathematical conditions can be verified.

\begin{acks}[Acknowledgments]
We are grateful to Dario Gasbarra for useful suggestions on our text.%The text of this article has been edited from a nearly completed manuscript of Esa Nummelin, dated August 14, 2017, by Elja Arjas, following a request made by Nummelin.

\end{acks}

\end{document}